\documentclass[a4paper, 12pt, reqno]{amsart}

\usepackage{fullpage} 
\usepackage{tcolorbox}
\usepackage{amsfonts, amsmath, amssymb, amsthm, hyperref, enumitem, fancyhdr, color, soul, comment,tikz, standalone, setspace}
\usepackage{amsaddr}
\usepackage[pagewise]{lineno}
\newtheorem{lemma}{Lemma}
\newtheorem{remark}{Remark}
\newtheorem{theorem}{Theorem}
\newtheorem{prop}{Proposition}
\newtheorem{definition}{Definition}
\newtheorem{corollary}{Corollary}
\allowdisplaybreaks
\title{ Global Attractor for the Periodic Generalized Korteweg-De Vries Equation Through Smoothing}
\author{Ryan McConnell}
\address{Department of Mathematics, University of Illinois, Urbana, IL 61801, USA}
\begin{document}
\begin{abstract}
   We establish a smoothing result for the generalized KdV (gKdV) on the torus with polynomial non-linearity, damping, and forcing that matches the smoothing level for the gKdV at $H^1$. As a consequence, we establish the existence of a global attractor for this equation as well as its compactness in $H^s(\mathbb{T})$, $s\in (1,2).$
\end{abstract}
\maketitle
\section{Introduction}
This paper is concerned with the g-KdV of the form 
\begin{equation}\label{gKdV}
\begin{cases}
    u_t + u_{xxx} + (g(u))_x = 0,\,\,\, x\in\mathbb{T},\,\,t\in\mathbb{R}\\
    u(x,0) = u_0\in H^s,
\end{cases}
\end{equation}
where we will be interested in, specifically, the case when $u$ is real and $g(u) = P(u)$ for some real polynomial. For $\mathbb{K} := \mathbb{T}$ or $\mathbb{R}$, it's known that
\begin{equation*}
    \int_\mathbb{K} u\, dx,
\end{equation*}
\begin{equation*}
    \int_\mathbb{K} |u|^2\,dx,
\end{equation*}
and
\begin{equation*}
    E(u) = \frac{1}{2}\|u_x\|^2_2 - \int_\mathbb{K} G(u) \,dx,
\end{equation*}
where $G(0) = 0$ is a primitive of $g$ are all conserved quantities of \eqref{gKdV}.

The local and global well-posedness of \eqref{gKdV} has been extensively studied over the last 25 or so years. In particular, integrability methods which are applicable for the KdV and mKdV via the Miura transformation defined by 
\[
\mathbf{M}v = -6(v_x+v^2),
\]
which takes solutions of the mKdV equation to those of the KdV equation,
give global well-posedness for the KdV in $H^{-1}$ as well as for the mKdV in $H^0 = L^2$ \cite{killip2019kdv}. Additionally, work done at the turn of the century by Bourgain, Colliander, Keel, Staffilani, Takaoka, and Tao using fixed point arguments in Bourgain spaces and, more specifically, the I-method, (see: \cite{bourgain1993fourier, colliander2001global, colliander2003sharp, colliander2004multilinear, colliander2007resonant}), a general method for establishing global well-posedness below the energy level, established that the KdV was globally well-posed on $\mathbb{T}$ for $s \geq -1/2$ and on $\mathbb{R}$ for $s > -3/4$. Similarly, they showed the analogous result for the mKdV with $s \geq 1/2$ on $\mathbb{T}$ and $s > 1/4$ on $\mathbb{R}$.

It's known (see: \cite{colliander2003sharp, farah2010supercritical, oh2020smoothing} and sources within) that, for small data and polynomial non-linearity $g(x) = u^{k+1}$, (\ref{gKdV}) is globally well-posed below the $H^1$ level on $\mathbb{T}$. Additionally, it's similarly known that in the defocusing case (negative sign in front of the derivative non-linearity) with even $k$ we again recover global well-posedness regardless of the size of the initial data. In particular, they establish global existence results for the supercritical defocusing gKdV  in $H^s(\mathbb{T})$ for $s > \frac{13}{14} - \frac{2}{7k}$. On the real line it's been shown (\cite{farah2010supercritical} and sources within) we have global existence in the supercritical case for $s > \frac{k-4}{2k}$ given small initial data.

Consider now the dissipative gKdV on $\mathbb{K} = \mathbb{R}$ or $\mathbb{T}$ given by 
\begin{equation}\label{AttractorPDEoNR}
    \begin{cases}
    u_t + u_{xxx} +\gamma u+ (g(u))_x = f(x) ,\,\,\, x\in\mathbb{\mathbb{K}},\,\,t\in\mathbb{R}\\
    u(x,0) = u_0\in {H^s},
\end{cases}
\end{equation}
with $\gamma > 0$ and $f\in H^s$ time independent. When $\mathbb{K} = \mathbb{R}$, it's been shown by Dlotko, et al. \cite{DLOTKO20093934} that \eqref{AttractorPDEoNR} with forcing in ${H}^1$ and $g\in C^{2}$ with Lipschitz second derivative as well as additional growth assumptions at $0$ and infinity possesses a global attractor in $H^1$. Similarly, Wang et al. \cite{wang2015global} established the existence of a global attractor in ${H}^2$ for $g(x) = x^4/4$ given $f\in {H}^2$ time independent, with a refinement by Wang \cite{wang2012long} establishing the existence in ${H}^s$ given $f\in L^2\cap {H}^{s-3}.$ 

We now restrict ourselves to the case that $\mathbb{K} = \mathbb{T}$, where $\dot{H}^s$ is the space of mean-zero $f$ in $H^s$, $\gamma > 0$, and $f\in \dot{H}^1$ time independent. Goubet \cite{goubet2000asymptotic} showed that \eqref{AttractorPDEoNR} in $\dot{L}^2$ has a global attractor in $\dot{L}^2$ that is compact in $H^3$. Tsugawa  \cite{tsugawa2004existence} then applied the I-method to establish the existence of a global attractor for \eqref{AttractorPDEoNR} in $\dot{H}^s$ with $s > -3/8$ and forcing $f\in \dot{L}^2$, as well as showing that the attractor is compact in $H^3$. Following the techniques used in refining the I-method (adding correction terms to the modified energy functional), Yang \cite{yang2011global} was able to improve the $s$ bound to $s \geq -1/2$, coinciding with the local existence theory for the KdV. Later, Erdo\u{g}an and Tzirakis showed in \cite{erdogan2011long} using smoothing that \eqref{AttractorPDEoNR} has a global attractor in $\dot{L}^2$ that is compact in $H^s$, $s\in (0,1)$, with forcing in $\dot{L}^2.$ Moving on to the mKdV, Goyal \cite{goyal2018global} proved an analogous statement to those of Tsugawa and Yang by proving that the forced and damped mKdV possesses a global attractor in $\dot{H}^s$ for $1 > s > 11/12$ with forcing $f\in\dot{L}^2\cap H^1$.

Following ideas presented in \cite{erdougan2016dispersive}, we seek to prove the following.

\begin{theorem}\label{Theorem3}
 Let $g, H, Q$ be polynomials with $g(0) = 0$ and
 \[
 g(x) = H(x) + Q(x),
 \] where $\deg H\leq 4$, $Q\equiv 0$ or $\deg Q \geq 5$ with $n$ odd with negative leading coefficient \footnote{That is, $g(x) = \sum_{k=0}^4 a_kx^k + \sum_{k=5}^n a_k x^k$ where $n \geq 5$ is either odd with $a_n < 0$, or $a_5=\cdots=a_n=0$.}.  Then (\ref{AttractorPDEoNR}) with $\mathbb{K} = \mathbb{T}$, $\gamma > 0$, and $f\in\dot{H^1(\mathbb{T})}$ is globally well-posed and the data-to-solution map, $S_t$, has a global attractor  $A\subset \dot{H^1}$ that is compact in $H^s$ for $s\in(1,2).$
\end{theorem}

    In order to prove Theorem \ref{Theorem3}, we follow the method in \cite{oh2020smoothing}, in which they proved a local smoothing result for the gKdV and polynomial bounds on higher order Sobolev norms. In particular, we adopt their analysis of the resonant set described in the beginning of Section 3, as well as Lemmas \ref{GeneralLemma w Hn} and \ref{GeneralLemma wo Hn}-- although we have changed the statement to Lemma \ref{GeneralLemma w Hn} and the proofs to both. We overcome obstacles involved in the application of the guage transformation to the forcing term, as employed in \cite{goyal2018global}, at the cost of a much longer local existence proof. We then apply a normal form transformation and perform an analysis on the resulting terms. While we do employ a normal form transformation similar to \cite{oh2020smoothing}, we change the restrictions associated to the transformation. The benefit of this is simplified proofs, at the cost of somewhat nasty subscripts on the summations. In total, however, we are able to match the level of smoothing known for the gKdV at the $H^{1}$ level, and then use this as well as Rellich's theorem to conclude the existence of a global attractor.

\begin{remark}
It's unclear if the proof of Goubet, \cite{goubet2000asymptotic}, would extend to the gauge gKdV that we work with, whereas the result follows readily from smoothing with a modification for local well-posedness.
\end{remark}
    
\begin{remark}
The method we employ to close the contraction for our gauged, forced, and damped equation is of independent interest. It appears that such a method will be useful in handling resonance interactions of the form $k_j = k$ for some internal $k_j$, allowing one to systematically study unconditional well-posedness for the KdV heirarchy. Specifically, it will enable one to handle ``linear'' groups of the form
\[
\partial_tu + \partial_x^{2k+1}u +\alpha_{2k-1}(t, u)\partial_x^{2k-1}u + \cdots +\alpha_{1}(t,u)\partial_x u,
\]
where $\alpha_j(t,u)$ are $x$-independent functions of $t$ and $u$. 
\end{remark}


\section{Preliminaries}
Let $k$ be the dual Fourier variable to $x\in\mathbb{T}$ and $\tau$ be the dual Fourier variable to $t\in\mathbb{R}$. We denote the Fourier transform of a function $u$ on the torus $\mathbb{T}$ by $\widehat{u}(k)$ or $\widehat{u}_{k}$. When $u$ is a spacetime function on $\mathbb{T}\times\mathbb{R}$, we denote both the space-time Fourier transform and the spacial Fourier transform as $\widehat{u}_k$, since confusion is mitigated by including the variables. 

We write $A\lesssim_{\varepsilon} B$ when there is a constant $0 < C(\varepsilon)$ such that $|A|\leq C|B|$; $A\gg_\varepsilon B$ to be the negation of $A\lesssim_\varepsilon B$; and $A\sim_\varepsilon B$ if, in addition, $B\lesssim_\varepsilon A$. 

Define $\langle h\rangle := (1+|h|)$ and $J^s$ to be the Fourier multiplier given by $\langle k\rangle^s$. Let $W_t = e^{-t\partial_x^3}$ be the propagator for the Airy equation and $W_t^\gamma = e^{-t\partial^3_x-\gamma t}$ be the propagator for the damped Airy equation with damping coefficient $\gamma > 0$. We then define the Bourgain space $X^{s,b}$ first introduced in the seminal paper \cite{bourgain1993fourier}, by
\[
    \|u\|_{X^{s,b}} = \|\langle k \rangle^{s}\langle \tau - k^3\rangle^{b}\widehat{u}\|_{L^2_\tau\ell^2_k} = \|\langle k\rangle ^s\langle \tau\rangle^b \widehat{u}(k, \tau + k^3)\|_{\ell^2_k L^2_\tau},
\]
which measures, in a sense, how much $u$ deviates from the free solution $W_tu$. Given $u_0\in H^s$, we see that $\|W_tu_0\|_{X^{s,b}}$ is infinite, so for $T > 0$ we define the restricted Bourgain space as the space of equivalence classes of functions endowed with the norm
\begin{equation*}
\|u\|_{X^{s,b}_T} := \inf\{\|v\|_{X^{s,b}}\,|\,u|_{[0,T ]} = v|_{[0,T]}\},
\end{equation*}
with dual space $X^{-s,-b}.$ Because of this natural pairing, we will often invoke duality and hence it will be natural to define the hyper-plane
\[
\Gamma_n := \{\tau+\tau_1 + \cdots +\tau_n = 0,\,\,k+k_1+\cdots +k_n = 0\},
\]
with obvious inherited measure denoted by $d\Gamma$.
This Bourgain space is almost good enough, but it's known (see: \cite{kenig1996bilinear}) that the bilinear estimate 
\[
\|\partial_x(uv)\|_{X^{s,-b_1}_T}\lesssim \|u\|_{X^{s,b_2}_T}\|v\|_{X^{s,b_2}_T}
\]
fails for any $b_1 > 0$, $b_2\ne \frac{1}{2}$, and $b_1+b_2\leq 1$. However, the $X^{s,1/2}$ norm barely fails to control the $C^0_tH^s_x$ norm, so we have to define the modified Bourgain space $Y^s$ by
\begin{equation}\label{Free Soln in Ys}
\|u\|_{Y^s} := \|u\|_{X^{s,1/2}} + \|\langle k\rangle^s \widehat{u}\|_{\ell^2_k L^1_\tau},
\end{equation}
and the auxiliary space $Z^s$ by
\[
\|u\|_{Z^s}:= \|u\|_{X^{s, -1/2}} + \left\| \frac{\langle k\rangle^s}{\langle \tau - k^3\rangle}\widehat{u}(k, \tau)\right\|_{\ell^2_k L^1_t}.
\]
Analogously to the way we defined restricted norm  $X_T^{s,b}$, we define $Y^s_T$ and $Z^s_T$. It's well known that with these modifications $Y^s\hookrightarrow C^0_tH^s$, for reference see: \cite{erdougan2016dispersive}.

We now record some facts (see: \cite{colliander2003sharp, erdougan2016dispersive}):
\begin{prop}
For any $\chi\in \mathcal{S}_t$ and $f\in C^\infty_x(\mathbb{T})$, $F\in Z^s$,
\begin{align*}
    \|\chi(t)e^{-t\partial_x^3}f\|_{Y^s}&\lesssim \|f\|_{H^s}\\
    \left\|\chi(t)\int_0^te^{-(t-s)\partial_x^3}F(s)\,ds\right\|_{Y^s}&\lesssim \|F(s)\|_{Z^s}.
\end{align*}
\end{prop}
Additionally, we have the following linear estimates:
\begin{align*}
    \|\chi(t)v\|_{L^4_{x,t}}&\lesssim\|u\|_{X^{0,1/3}}, &\|\chi(t)v\|_{L^6_{x,t}}&\lesssim_{\delta}\|v\|_{X^{\delta, 1/2+\delta}},
\end{align*}
and
\[
    \|\chi(t/T)v\|_{X^{s,b}}\lesssim T^{b'-b}\|v\|_{X^{s,b'}}\,\,\mbox{for} -\frac{1}{2} < b < b' < \frac{1}{2}.
\]
For $2 < q < 6$ interpolation between the $L^6$ and the trivial estimate gives the existence of an $\varepsilon > 0$ such that \begin{equation}\label{InterpolationIneq}\|\chi(t)v\|_{L^q_{x,t}}\lesssim_{\varepsilon}\|v\|_{X^{\varepsilon, 1/2-\varepsilon}}.\end{equation}
We also have (See: \cite{bao2013global})
\begin{align*}
    \|\chi(t)v\|_{L^\infty_{x,t}}&\lesssim_\varepsilon \|v\|_{Y^{1/2+\varepsilon}},\\
\end{align*}
and (7.1) in \cite{colliander2004multilinear}: there exists $\delta > 0$ so that for all $u_i$ with support of $\widehat{u}_i$ in $\mathbb{T}\times[0,T]$, we have
\begin{equation}\label{TaoEstiamte}
\|u_1\cdots u_n\|_{L^2_{x,t}}\lesssim_\delta \|u_1\|_{X^{0,1/2-\delta}}\|u_2\|_{X^{0, 1/2-\delta}}\prod_{i=3}^n \|u_i\|_{X^{1/2-\delta, 1/2-\delta}}.
\end{equation}

\begin{remark}
With the modified linear group $W^\gamma_t$, it follows that 
\[
\|\chi(t)W_t^\gamma f\|_{Y^s}\lesssim \|f\|_{H^s}.
\]
See (\cite{erdogan2011long} Lemma 3.3) and (\cite{erdougan2016dispersive}, Remark 3.13 \& Lemma 3.15). The key fact is that $\chi(t)e^{-t\gamma}$ is a Schwarz function for any smooth cutoff $\chi$.
\end{remark}

The following lemma will be of use for proving Theorem \ref{Theorem3}:
\begin{lemma}\label{AbsorbingSet}
Consider, $u$, a solution to \eqref{AttractorPDEoNR} with $\mathbb{K} = \mathbb{T}$. Then $\|u\|_2$ and $\|u_x\|_2$ are bounded a priori.
\end{lemma}
\begin{proof}
The first part of this is well known by differentiating $\|u\|_2^2:$
\begin{align*}
    \partial_t \|u\|_2^2 &= -2\gamma\|u\|_2^2 + \int_\mathbb{T} uf\,dx\\
    &\leq -\gamma\|u\|_2^2 + C(\|f\|_{L^2}).
\end{align*}
Applying Gronwall's inequality completes the first part of the proof.

The case $g(x) = a_2x^2+a_3x^3+a_4x^4$ is handled in (\cite{DLOTKO20093934}, Lemma 4), as $g(x)\lesssim |x|^3 + |x|^5$.

Assuming $g(x) = \sum_{k=2}^n a_ku^k$ with $n \geq 5$ odd and $0 > a_k$, $N > 0$ large enough, and differentiating $E(u)$, we have
\begin{align*}
\partial_t E(u) &= \int_\mathbb{T} u_x\left( f_x - \sum_{k=2}^n a_k\partial_{xx}(u^{k}) - \gamma u_x - u_{xxxx}\right)\\
&\qquad\qquad\qquad- \sum_{k=2}^n a_k u^{k}\left(f - \sum_{j=2}^na_j\partial_x(u^{j}) - \gamma u - u_{xxx}\right)\,dx\\
&=-\gamma E(u) + \int_\mathbb{T} u_xf_x - \sum_{k=2}^n a_ku^{k}f+ \sum_{k=1}^n\frac{a_k\gamma(k-1)}{k+1} u^{k+1}\,dx\\
&\lesssim -\gamma E(u) + \frac{\gamma}{4}\|u_x\|_2^2 + \int_\mathbb{T}\sum_{k=2}^n \frac{|a_k|}{N}u^{n+1} + \sum_{k=1}^{n-1} \frac{|a_k|\gamma(k-1)}{N(k+1)}u^{n+1}\,dx\\
&\qquad\qquad\qquad  - \int_\mathbb{T}\frac{|a_n|\gamma(k-1)}{k+1}u^{n+1}\,dx + C(f, N, \gamma)\\
&=-\frac{1}{2}\gamma E(u) + \int_\mathbb{T}\sum_{k=2}^{n} \frac{\gamma a_k}{2(k+1)}u^{k+1} + \sum_{k=2}^n \frac{|a_k|}{N}u^{n+1} + \sum_{k=1}^{n-1} \frac{|a_k|\gamma(k-1)}{N(k+1)}u^{n+1}\,dx\\
&\qquad\qquad\qquad- \int_\mathbb{T}\frac{|a_n|\gamma(n-1)}{n+1}u^{n+1}\,dx+ C(f, N, \gamma)\\
&\lesssim -\frac{1}{2}\gamma E(u) + \int_\mathbb{T}\frac{C(|a_2|, \cdots, |a_{n-1}|, \gamma)}{N}u^{n+1} - \frac{\gamma |a_n|n}{2(n+1)}u^{n+1}\,dx + C(f, N, \gamma)\\
&\leq -\frac{1}{2}\gamma E(u) + C(f, N, \gamma),
\end{align*}
for $N$ large enough. Note that we have only used Young's inequality, the definition of $E(u)$ to replace $\gamma/4\|u_x\|_2^2$, and the fact that we're on $\mathbb{T}$.
The claim then follows from Gronwall's by \[\|u_x\|^2_{2}\lesssim E(u).\]

\end{proof}
\begin{remark}
It follows from Lemma \ref{AbsorbingSet} that $\limsup_t \|u\|_{H^1_x}\leq C(\|f\|_{H^1})$, and hence that there is an absorbing set in $H^1$. 
\end{remark}

%
%
\section{Set-up}
We first note that if $u$ is mean-zero, then, at least formally, we have the relation:
\begin{align*}
    \partial_t \int u\,dx &= -\int \gamma u\,dx. 
\end{align*}
It follows that if $u_0$ has mean-zero, then so must $u$ for all time. Additionally, we have that the transformation $u\mapsto u(t, x-ct)$ removes any linear term in the polynomial non-linearity in (\ref{AttractorPDEoNR}). We thus assume that $g$ contains only terms of degree two and up, and that we can assume that we are always working with mean-zero functions.

Now, we first seek to prove local existence of the forced and damped equation, encapsulated in the following proposition.

\begin{prop}\label{SmootherfTheorem}
Consider \eqref{AttractorPDEoNR} with $\mathbb{K} = \mathbb{T}$:
\begin{equation}\label{SmootherfPDE}
    \begin{cases}
    u_t + u_{xxx} +\gamma u+ \left(g(u)\right)_x = f(x) ,\,\,\, x\in\mathbb{T},\,\,t\in\mathbb{R}\\
    u(x,0) = u_0\in \dot{H}^1
\end{cases}
\end{equation}
for $f\in \dot{H}^{1}$ time-independent, $\gamma > 0$, and $g\in \mathbb{R}[x].$ Then there is a $T = T(\gamma, \|f\|_{H^{1}}, \|u_0\|_{H^{1}})$ such that, on $[0, T]$, \eqref{SmootherfPDE} has a unique solution, $u$, given by the translate \eqref{Translate version of u}. Furthermore, we have that, on $[0,T]$, $u$ satisfies
\[ \|u\|_{C^0_tH^1_x}\lesssim \|L_t[u]u\|_{Y^1_T} = \|\tilde{u}\|_{Y^1_T}\leq C,\]
where $C = C(\gamma, \|f\|_{H^{1}}, \|u_0\|_{H^{1}}),$ with $L_t$ and $\tilde{u}$ defined at \eqref{L definition}.
\end{prop}
\begin{remark}
The non-dissipative g-KdV considered in \cite{colliander2004multilinear} is shown to be locally well-posed in $H^s$, so this proposition is, in a sense, obvious. However, the application of $L_t$ modifies the forcing term to be $u$ dependent, which makes closing the contraction non-trivial.

Moreover, the local well-posedness bound will be important in the later sections.
\end{remark}

For simplicity we will assume that the non-linearity in $\eqref{SmootherfPDE}$ has only two terms given by $\partial_x(au^n + bu^m)$, as this will be enough to generalize later. Passing through the Fourier transform we see that we may write the homogenous non-linearity as 
\[a\sum_{k_1+\cdots + k_n = k}ik\prod_{j=1}^n \widehat{u}_{k_j} + b\sum_{k_1 + \cdots + k_m = k}ik\prod_{j=1}^m \widehat{u}_{k_j}.\]

For us, we will say resonance\footnote{In the literature this is frequently defined differently. Specifically, it is defined to be when $H_n = 0,$ as defined in \eqref{Hn def}.} occurs when an internal frequency equals an external frequency. Using this, we decompose the above into $\mathcal{R}[u]$ and $\mathcal{NR}[u]$ given by 
\[ \widehat{\mathcal{R}}[u]_k = a\sum_{\substack{k_1+\cdots + k_n = k\\ \exists j_0\,:\,k_{j_0} = k}}ik\prod_{j=1}^n \widehat{u}_{k_j} + b\sum_{\substack{k_1+\cdots + k_m = k\\ \exists j_0\,:\,k_{j_0} = k}}ik\prod_{j=1}^m \widehat{u}_{k_j}.\]

We write $\mathcal{R} = \mathcal{R}^1 + \mathcal{R}^2$, where, by symmetry,
\begin{align*}
     \widehat{\mathcal{R}^1}[u]_k &:= a\sum_{r=1}^n\sum_{\substack{k_1+\cdots + k_n = k\\ k_{r} = k}}ik\prod_{j=1}^n \widehat{u}_{k_j} + b\sum_{r=1}^m\sum_{\substack{k_1+\cdots + k_m = k\\k_{r} = k}}ik\prod_{j=1}^m \widehat{u}_{k_j}\\
     &=ik\widehat{u}_k\left(an\sum_{k_1+\cdots +k_{n-1}=0}\prod_{j=1}^{n-1}\widehat{u}_{k_j} + bm\sum_{k_1+\cdots +k_{m-1}=0}\prod_{j=1}^{m-1}\widehat{u}_{k_j}\right)\\
     &=ik\widehat{u}_k\left(an\int_\mathbb{T} u^{n-1}(x,t)\,dx + bm\int_\mathbb{T} u^{m-1}(x,t)\,dx\right).
\end{align*}
$\mathcal{R}^2$ is then defined to be what remains of $\mathcal{R}$. In particular, $\mathcal{R}^2$ has at least two internal frequencies equal to the external frequency, $k$.

Define the Fourier multiplier 
\begin{equation}\label{L definition}
\widehat{L_t[v]}_k := \exp\left(ik \int_0^t\int_\mathbb{T} an v^{n-1}(x,s) + bm v^{m-1}(x,s)\,dxds\right)
\end{equation}
and denote 
\[
\tilde{u}= \mathcal{F}^{-1}\left(\widehat{L_t[u]}\widehat{u}\right), \qquad\qquad\tilde{f} = \mathcal{F}^{-1}\left(\widehat{L_t[\tilde{u}]}\widehat{f}\right).
\]
Notice that, using the exponential nature of the multiplier, we have:
\[
\widehat{(\tilde{u}^n)}_k = \sum_{k_1+\cdots +k_n=k}\prod_{i}\widehat{\tilde{u}}_{k_i} = \sum_{k_1+\cdots +k_n=k}\prod_{i}\widehat{L_t[u]}_{k_i} \widehat{u}_{k_i} = \widehat{L_t[u]}_k\widehat{(u^n)}_k,
\]
and hence the zero modes are equal. It follows that if $u$ solves $(\ref{SmootherfPDE}),$ then $\tilde{u}$ solves
\begin{equation}\label{modified LWP PDE}
    \begin{cases}
    \tilde u_t + \tilde{u}_{xxx} +\gamma \tilde{u}+ \mathcal{R}^2[\tilde{u}] + \mathcal{NR}[\tilde{u}]=\tilde{f}(x,t)\\
    \tilde{u}(x,0) = u_0\in H^s\,\,\,x\in\mathbb{T},\,\,t\in\mathbb{R},
    \end{cases}
\end{equation}
and if $\tilde{u}$ solves $(\ref{modified LWP PDE}),$ then a simple translate of $\tilde{u}$,
\begin{equation}\label{Translate version of u}
u(x,t) = \tilde{u}\left(x + \int_0^t\int_\mathbb{T} an\tilde{u}^{n-1}(y,s)+bm\tilde{u}^{m-1}(y,s)\,dy\,ds, t\right),
\end{equation}
solves $(\ref{SmootherfPDE}).$
%
%

\subsection{Dispersion Relation}
Let 
\begin{equation}\label{Hn def}
    H_n:= H_n(k_1,\cdots, k_n) = \left(\sum_{j=1}^n k_j\right)^3 - \sum_{j=1}^n k_j^3.
\end{equation}

We then have the following proposition: 
\begin{prop}[\cite{oh2020smoothing}, Proposition 2] Consider $k = \sum_{j=1}^n k_j$ for $k_j\in\mathbb{Z}^*$ and $H_n$ as defined above. Denote $k^*:=\max\{|k_1|,\cdots, |k_n|\}$ and $k^*_j$ be the $j$-th largest term in $\{|k_1|,\cdots, |k_n|\}$. Then the following are true:
\begin{itemize}
    \item[(1)] If $n = 2$ then $H_2\gtrsim (k^*)^2$.
    \item[(2)] If $n = 3$ then at least one of the following is true:
        \begin{itemize}
            \item[A.] $H_3\gtrsim (k^*)^2$.
            \item[B.] $k_{j_0} = k$ for some $j_0\in\{1,2,3\}.$
            \item[C.] ${k_j}\gtrsim k$ for all $j\in\{1,2,3\}.$
        \end{itemize}
    \item[(3)] If $n \geq 4$ at least one of the following is true:
        \begin{itemize}
            \item[A.] $H_n\gtrsim (k^*)^2$.
            \item[B.] $k_{j_0} = k$ for some $j_0\in\{1, \cdots, n\}.$
            \item[C.] ${k^*_3}\gtrsim k$
            \item[D.] $(k^*_3)^2k^*_4\gtrsim (k^*)^2$
        \end{itemize}
\end{itemize}
\end{prop}

%
%

\subsection{Multilinear Operators}
Here we quickly introduce some notation that appears frequently in what is to come, as well as several related lemmas that will be used time and time again. The generality employed in Lemmas \ref{GeneralLemma w Hn} and \ref{GeneralLemma wo Hn} is perhaps a bit much, but the language it's in will be useful in the following sections. Additionally, the vast majority of the estimates in the following sections follow from Lemma \ref{GeneralLemma wo Hn} and the Corollaries to Lemmas \ref{GeneralLemma w Hn} and \ref{GeneralLemma wo Hn}. 
\begin{definition}
Given a symbol $\sigma(k,\vec{k}) = \sigma(k, k_1, \cdots, k_n)$ with $n\geq 2$, a multilinear Fourier multiplier $T_\sigma^n$ is defined by 
\[
T_\sigma^n(u_1,\cdots, u_n) := \sum_{k\in \mathbb{Z}^*}e^{ikx}\sum_{(k_1, \cdots, k_n)\in \Omega_k} \sigma(k_1, \cdots k_n)\prod_{j=1}^n\widehat{u}_{k_j},
\]
where $\Omega_k$ restricts the frequency interactions and is associated with the operator $T.$
\end{definition}
\begin{remark}
Because of mean-zero conservation, every $\Omega_k$ associated to a $T^n_\sigma$ considered in this paper will be of the form $\Omega_k\subset (\mathbb{Z}^*)^n$.
\end{remark}
\begin{lemma}\label{GeneralLemma w Hn} Let $n\geq 2$, $0 < T\ll 1$, $s_0\in\mathbb{R}$, $s_1 > 1/2$, $u_i := \chi(t/T)u_i(t)$ for a smooth cut-off function $\chi$, and $T_\sigma^n$ a multilinear operator of the type above. If $\mbox{supp } \sigma(k,\vec{k})\subset \Gamma_n\cap\{H_n\gtrsim k^2\}$ and:
\begin{equation}\label{Symmetric Condition for Hn}
    \sup_{(k_1\cdots, k_n)\in \Omega_k}\frac{|k|^{s_0-1}|\sigma(k, k_{1},\cdots,k_{n})|}{(k^*k^*_2)^{s_1}} = O(1),
    \end{equation}
    then we have the bound
\[
\|T_\sigma^n(u_{1}, \cdots, u_{n})\|_{Z^{s_0}_T}\lesssim_{\varepsilon, n} T^\varepsilon \prod_{i=1}^n\|u_{i}\|_{Y^{s_1}_T},
\]
for some $\varepsilon > 0$.
\end{lemma}
\begin{proof}
$ $\newline 
This proof is essentially the same as in \cite{goyal2018global, oh2020smoothing}, but for the sake of completeness we will include it. We assume that $n\geq 4$ (for otherwise the bounds are strictly better), 
\[
|k_1|\geq \cdots \geq |k_n|,
\]
and that all Fourier transforms considered are non-negative.

We first look at the $X^{s_0,-1/2}_T$ portion, and split ourselves into two cases:
\begin{itemize}
    \item[(A1)] $ \tau - k^3\gtrsim \max_{1\leq r\leq n}\{ \tau - k^3,  \tau_{r} - k_{r}^3\}$
    \item[(A2)] $ \tau_{i} - k_{i}^3\gtrsim \max_{1\leq r\leq n}\{ \tau - k^3,  \tau_{r} - k_{r}^3\}$ for some $1\leq i\leq n$.
\end{itemize}
To prove the $X^{s_0,-1/2}$ estimate in Case A1, we notice that
\[
k^2\lesssim H_n\lesssim \sum_{r=1}^n \langle \tau_{r} - k_{r}^3\rangle + \langle \tau - k^3\rangle\lesssim \langle \tau - k^3\rangle,
\]
so that by duality and applying condition \eqref{Symmetric Condition for Hn}, it suffices to prove:
\[
\int_{\substack{\Gamma_n\\ H_n\gtrsim k^2}} |k|\widehat{z}_k (k_1^*k_2^*)^{s_1} \langle\tau - k^3\rangle ^{-1/2} \prod_{r=1}^n \widehat{u}_{k_{r}}\,d\Gamma \leq T^\varepsilon\|z\|_{L^2} \prod_{r=1}^n \|u_r\|_{X^{s_1,1/2}_T}.
\]
This, however, follows from Plancherel and H\"olders with $L^2$ on $z$ and applying $\eqref{TaoEstiamte}$ to the remaining terms in $L^2$. We then have
\begin{align*}
\|T^n_\sigma&(u_1,\cdots, u_n)\|_{X^{s_0, -1/2}}\\
&\lesssim \|u_{1}\|_{X^{s_1, 1/2-\delta}_T}\|u_{2}\|_{X^{s_1, 1/2-\delta}_T}\prod_{r=3}^n \|u_r\|_{X^{1/2-\delta, 1/2-\delta}_T}.   
\end{align*}
The result follows by summation over all index orderings and yields a small power of $T$.

To prove the $\ell^2_k L^1_\tau$ estimate, we note that, for $a > 0$, 
\begin{equation}\label{l2L1 trick}
\int_{\mathbb{R}} \frac{1}{(|\tau| + a)^2}\,d\tau\lesssim \frac{1}{a}.
\end{equation}
Now, since $\langle \tau - k^3\rangle \gtrsim H_n\gtrsim k^2$, we must have that $\langle \tau - k^3\rangle  \gtrsim \langle \tau - k^3\rangle + k^2$. It follows from H\"olders in $\tau$ and condition \eqref{Symmetric Condition for Hn} that
\begin{align*}
\left \| \frac{\langle k\rangle^{s_0} }{\langle \tau - k^3\rangle}\mathcal{F}_{x,t}\left(T^n_{\sigma\chi_{H_n\gtrsim k^2}}(u_1,\cdots, u_n)\right)\right\|_{\ell^2_k L^1_\tau}&\lesssim \left\|\langle k\rangle^{s_0-1} \mathcal{F}_{x,t}\left(T^n_{\sigma\chi_{H_n\gtrsim k^2}}(u_1,\cdots, u_n)\right)\right\|_{\ell^2_k L^2_\tau}\\
&\lesssim \left\|\mathcal{F}_{x,t}\left(T^n_{\langle k_1^*\rangle^{s_1}\langle k_2^*\rangle^{s_2}\chi_{H_n\gtrsim k^2}}(u_1,\cdots, u_n)\right)\right\|_{\ell^2_k L^2_\tau}\\
&\lesssim_\varepsilon T^\varepsilon\prod_{r=1}^n \|u_r\|_{X^{s_1, 1/2}_T},
\end{align*}
by the proof of the $X^{s_0,-1/2}$ bound after the simplification.

\textbf{Case A2}: Assume that $\langle \tau_{i} - k_{i}^3\rangle\gtrsim k^2$ is the largest, and note that this estimate will be the only part of this lemma that will require the $Y^s$ norm to explicitly appear on the right hand side, as opposed to the $X^{s,b}$ norm. 

For the $X^{s_0, -1/2}$ norm, we use duality again to reduce the problem to 
\begin{align*}
\int_{\Gamma_n}|k|^{s_0}\widehat{z}_k\langle \tau_{i} - k_{i}^3\rangle^{-1/2} \sigma(k,\vec{k})\prod_{r = 1}^n \widehat{u}_{k_{r}}\,d\Gamma &\lesssim T^\varepsilon\|z\|_{X^{0, 1/2}} \|u_i\|_{X_T^{s_1, 0}}\prod_{\substack{r=1\\r\ne i}}^n \|u_r\|_{Y^{s_1}_T}.
\end{align*}
This quantity reduces to 
\begin{equation}\label{A1finalQuantity}
\int_{\Gamma_n}|k|^{s_0-1}\widehat{z}_k\sigma(k,\vec{k})\prod_{r=1}^n \widehat{u}_{k_{r}}\,d\Gamma\lesssim \int_{\Gamma_n}(k^* k_2^*)^{s_1}\widehat{z}_k\prod_{r=1}^n \widehat{u}_{k_{r}}\,d\Gamma.
\end{equation}
First assume that $k^* = |k_{1}|$, $k^*_2 = |k_{2}|$, and that $k^*,k^*_2\ne |k_{i}|$, then by applying Plancherel, H\"olders, Bourgain space embeddings, and Sobolev embedding we get:
\begin{align*}
\eqref{A1finalQuantity}&\lesssim \|zJ^{s_1}_xu_{1}\|_{L^2_{x,t}}\|u_{i}J_x^{s_1}u_{2}\|_{L^2_{x,t}}\prod_{\substack{r=3\\r\ne i}}^n\|u_{r}\|_{L^\infty_{x,t}}\\
&\lesssim \|z\|_{L^4_{x,t}}\|u_{i}\|_{L^2_tL^\infty_x}\|J^{s_1}_xu_{1}\|_{L^4_{x,t}}\|J^{s_1}_xu_{2}\|_{L^\infty_tL^2_x}\prod_{\substack{r=3\\r\ne i}}^n\|u_{r}\|_{L^\infty_{x,t}}\\
&\lesssim \|z\|_{X^{0,1/3}}\|u_i\|_{X^{1/2+\delta, 0}_T}\|u_{1}\|_{X^{s_1,1/3}_T}\|u_{2}\|_{Y^{s_1}_T}\prod_{\substack{r=3\\r\ne i}}^n\|u_r\|_{L^\infty_{x,t}}\\
&\lesssim T^\varepsilon\|z\|_{X^{0, 1/2}} \|u_i\|_{X_T^{s_1, 0}}\prod_{\substack{r=1\\r\ne i}}^n \|u_r\|_{Y^{s_1}_T}.
\end{align*}

Assuming now that $k_1^* = |k_{i}|$ (the case $k_2^*=|k_{i}|$ is identical), we use a simpler Plancherel and H\"olders application:
\begin{align*}
    \eqref{A1finalQuantity}&\lesssim \|J^{s_1}_xu_{i}\|_{L^2_{x,t}}\|zJ^{s_1}_xu_{2}\prod_{r=3}u_r\|_{L^2_{x,t}}\\
    &\lesssim \|z\|_{X^{0,1/2-\delta}}\|u_i\|_{X_T^{s_1,0}}\|u_{2}\|_{X_T^{s_1, 1/2-\delta}}\prod_{r=3}^n \|u_r\|_{X^{1/2-\delta,1/2-\delta}_T},
\end{align*}
where we conclude the result by trivial embeddings and time localization to gain a small power of $T$.

The only remaining piece is the $\ell^2_kL^1_\tau$ norm, but for $0 < \delta\ll 1$ we apply H\"olders in $\tau$ to reduce to
\[
\left\| \frac{\langle k\rangle^{s_0}}{\langle \tau - k^3\rangle^{1/2 - \delta}}\mathcal{F}_{x,t}\left( T^n_{\sigma\chi_{H_n\gtrsim k^2}}(u_1,\cdots, u_n)\right)\right\|_{L^2_{k,\tau}}.
\]
Duality in the same manner as above leaves us with having to show
\begin{align*}
\int_{\Gamma_n} |k|^{s_0}\sigma(k,\vec{k})\widehat{z}(k,\tau)\langle \tau_{i} - k_{i}^3\rangle ^{-1/2} \prod_{r=1}^n\widehat{u}_{k_{r}}\,d\Gamma &\leq T^\varepsilon\|z\|_{X^{0, 1/2-\delta}} \|u_i\|_{X_T^{s_1, 0}}\prod_{\substack{r=1\\r\ne i}}^n\|u_r\|_{Y_T^{s_1}}.
\end{align*}

Again using that $\langle \tau_{i}-k^3_{i}\rangle\gtrsim k^2$ and Condition \eqref{Symmetric Condition for Hn}, we reduce to having to showing the same bound for the integral
\[
\int_{\Gamma_n}\widehat{z}_k( k^*_1k^*_2)^{s_1}\prod_{\substack{r=1}}^n\widehat{u}_{k_{r}}\,d\Gamma.
\]
This bound follows in exactly the same way as we did the $X^{s_0, -1/2}$ bound for Case A2 above, because both $X^{0,1/3}\hookrightarrow L^4$ and \eqref{TaoEstiamte} have at least $\delta$ space for localization. Since we have other terms with room for localization, we again have a (slightly smaller) power of $T$.
\end{proof}
\begin{lemma}\label{GeneralLemma wo Hn} Let $n\geq 2$,  $0 < T\ll 1$, $s_0\in\mathbb{R}$, $s_1> 1/2$, $0 < \delta\ll 1$, $\varepsilon > 0$, $u_i := \chi(t/T)u_i(t)$ for a smooth cut-off function $\chi$, and $T_\sigma^n$, a multilinear operator of the type above. If:
    \begin{equation}\label{Symmetric Condition wo Hn}
    \sup_{(k_1\cdots, k_n)\in \Omega_k}\frac{|k|^{s_0}|\sigma(k, \vec{k})|}{(k^*)^{s_1-\delta}(k^*_2k^*_3k^*_4)^{s_1}} = O(1),
    \end{equation}
     with the convention $k^*_j = 1$ for all $j > n,$ then
\[
\|T_\sigma^n(u_{1},\cdots, u_{n})\|_{Z^{s_0}_T}\lesssim_{\varepsilon,\delta, n} T^\varepsilon \prod_{i=1}^n\|u_{i}\|_{Y^{s_1}_T}.
\]
\end{lemma}
\begin{proof}
$ $\newline
Assume the same conventions as in Lemma \ref{GeneralLemma w Hn}. In this regime we again invoke duality, Plancherel, and \eqref{Symmetric Condition wo Hn} to see that we must bound
\begin{equation}\label{FourierInt}
\int_{\Gamma_n}|k^*|^{s_1-\delta}(k^*_2k^*_3k^*_4)^{s_1}\widehat{z}_k\prod_{r=1}^n\widehat{u}_{k_r}\,d\Gamma,
\end{equation}
for $z\in X^{0,1/2}.$ From here on, we assume that $n\geq 4$, for otherwise the estimates are strictly better. Now, splitting the region we assume again that $|k_1|\geq \cdots \geq |k_n|$, it follows, by trading up powers of $k^*$ for lower powers of $k^*_j$ for $j = 2, 3, 4$ and applying Plancherel with H\"olders, that:
\begin{align}\label{Condition 2 Xsb estimate}
    (\ref{FourierInt})&\lesssim \|z\|_{L^4_{x,t}}\|J^{s_1 - \delta/4}_x u_{1}\|_{L^{16/3}_{x,t}}\|J_x^{s_1 - \delta/4} u_{2}\|_{L^{16/3}_{x,t}}\|J_x^{s_1 - \delta/4} u_{3}\|_{L^{16/3}_{x,t}}\\
    &\qquad\qquad\qquad\times\|J_x^{s_1 - \delta/4} u_{4}\|_{L^{16/3}_{x,t}}\prod_{r=5}^n\|u_{r}\|_{L^\infty_{x,t}}\nonumber\\
    &\lesssim \|z\|_{X^{0, 1/3}}\|u_{1}\|_{X^{s_1, 1/2-\delta/4}}\|u_{2}\|_{X^{s_1, 1/2-\delta/4}}\|u_{3}\|_{X^{s_1, 1/2-\delta/4}}\nonumber\\
    &\qquad\times\|u_{4}\|_{X^{s_1, 1/2-\delta/4}}\prod_{r=5}^n\|u_{r}\|_{L^{\infty}_{x,t}},\nonumber
\end{align}
with time-localization and $Y^{s}\hookrightarrow L^\infty_tH^s_x$ concludes this portion of the proof.

The final estimate involving the $\ell^2_k L^1_\tau$ is nearly the exact same as before, but here we split off $2/3$ powers of $\langle \tau - k^3\rangle$, reducing it to an $X^{s_0, -1/3}$ estimate, which, by duality, is of the form
\[
\int_{\Gamma_n}|k^*|^{s_1-\delta}|k^*_2k^*_3k^*_4|^{s_1}\widehat{z}_k\prod_{r=1}^n\widehat{u}_{k_{r}}\,d\Gamma,
\]
for $z\in X^{0, 1/3}$. Applying H\"olders in exactly the same way as in the $X^{s_0, -1/2}$ case, and noting that the $z$ bound is in $X^{0,1/3}$, we arrive at (\ref{Condition 2 Xsb estimate}) with a small power of $T$.

\end{proof}

The two following lemmas are useful variants of the above and all follow from the same proof:
\begin{corollary}\label{Homogenous Estimate} Under the assumptions of lemma \ref{GeneralLemma w Hn}, we have
\[
\|T_\sigma^n(u, \cdots, u)\|_{Z^{s_0}_T}\lesssim_{\varepsilon, n} T^\varepsilon \|u\|_{Y^{s_1}_T}^{n}.
\]
Furthermore, under the assumptions of lemma \ref{GeneralLemma wo Hn}:
\[
\|T_\sigma^n(u, \cdots, u)\|_{Z^{s_0}_T}\lesssim_{\varepsilon, n} T^\varepsilon \|u\|_{Y^{s_1}_T}^{n}.
\]
\end{corollary}
\begin{corollary}\label{Large Term Estiamte}
Under the assumptions of lemma \ref{GeneralLemma w Hn} with the modified \eqref{Symmetric Condition for Hn}:
\[
\frac{|\sigma(k,\vec{k})|}{|k|(k^*_2)^{s_1}} = O(1),
\]
and the additional assumption that $k_1 \gtrsim k^*$, we have 
\[\|T_\sigma^n(v, u_1, \cdots, u_{n-1})\|_{Z^{s_0}_T}\lesssim_{\varepsilon,n} T^\varepsilon\|v\|_{Y^{s_0}_T}\prod_{i=1}^{n-1}\|u_i\|_{Y^{s_1}_T}.
\]
\end{corollary}
\begin{proof}
$ $\newline
This follows from the proof of the above lemmas. The only important piece is that the additional assumption allows us to know the full $s_0$ regularity can be taken from $v$. 
\end{proof}

\subsection{Local Well-posedness of \ref{modified LWP PDE}}\label{LWP}

In this section we apply the idea used in \cite{goyal2018global} to prove well-posedness. We begin by proving that the non-linear Duhamel term is in $Y^1$. This is necessary because we will take the time here to strengthen some prior estimates in particular situations, and this will enable us to perform the technique.

Before getting to the idea, we note that (as will be seen in the Lemma \ref{forcing Contraction}) it's not at all obvious how one should go about proving the contraction in $Y^1$ with an $f\in H^1$. However, the non-linear part of the Duhamel term has $X^{1,-1/2}$ and $\|\langle k\rangle\langle \tau - k^3\rangle^{-1} \mathcal{F}_{x,t}(\cdot)\|_{\ell^2L^1_\tau}$ norms that are both bounded by powers of $\|\tilde{u}\|_{X^{1,1/2}}$, as opposed to powers of $\|\tilde{u}\|_{Y^1}$. This observation alone justifies doing the contraction in $X^{1, 1/2}$, since any bound on $\|\tilde{u}\|_{X^{1, 1/2}}$ implies a bound on $\|\Gamma_T[\tilde{u}]\|_{Y^1_T}$ (defined in the proof of well-posedness for \eqref{modified LWP PDE}) by Lemma \ref{Lemma: NONLINEAR TERM IN Y1}. Additionally, the closed balls in $X^{s_1,1/2}$ are also closed with respect to $\|\cdot\|_{X^{s, 1/2}}$ for $0\leq s\leq s_1$, so that we may perform a fixed point argument in $X^{s_1, 1/2}$ using $\|\cdot\|_{X^{s, 1/2}}$ and recover the same well-posedness properties enjoyed by traditional approaches. For more information on the second point, see Kato (\cite{kato1975quasi}, Lemma 7.3).

The above two observations suggest we should attempt to do the contraction in $X^{1,1/2}$ with the $X^{0, 1/2}$ norm to compensate for the loss incurred on $f$ in closing the contraction. This is exactly the route we take in this section.

\begin{lemma}[\cite{oh2020smoothing}, Lemma 2]\label{Lemma: NONLINEAR TERM IN Y1}
There exists an $\varepsilon > 0$ such that for any $0 < T\ll 1$ we have:
\[\|\mathcal{R}^2[u]\|_{Z^1_T} + \|\mathcal{NR}[u]\|_{Z^1_T} \lesssim_\varepsilon T^\varepsilon \|u\|_{X^{1,1/2}_T}\left(\|u\|^{n-1}_{X^{1,1/2}_T} + \|u\|^{m-1}_{X^{1,1/2}_T}\right),\]
where $u = \chi(t/T)u(t)$.
\end{lemma}
\begin{proof}
$ $\newline
Here our $\sigma = ik$, so we proceed by case:

\textbf{Case A:} Case A is handled by Lemma \ref{GeneralLemma w Hn} Case A1 as we have at least one large frequency, so that Condition \eqref{Symmetric Condition for Hn} becomes:
\[
\frac{|ik|}{k_1^*} = O(1).
\]
All that remains is the Case A2, which we improve by noting that we only need $k_1^*$ to make this estimate. Looking back at the $X^{1, -1/2}$ proof of Lemma \ref{GeneralLemma w Hn} in Case A2 under the assumptions that $k^* = k_1$ and $\langle \tau_i - k_i^3\rangle\gtrsim k^2$, we see that we we need to bound:
\begin{equation}\label{u in Y1 case A ineq}
\int_{\Gamma_n}\langle k^*\rangle\widehat{z}_k\prod_{r=1}^n \widehat{u}_{k_r}\,d\Gamma.
\end{equation}
This bound differs slightly from that in Lemma \ref{GeneralLemma w Hn}, but we still apply Plancherel, H\"olders, and then \eqref{TaoEstiamte}. Specifically, when $i= 1$ we need to show
\[
\eqref{u in Y1 case A ineq}\lesssim T^\varepsilon\|z\|_{X^{0,1/2-\delta}}\|u\|_{L^2_{x,t}}\|u\|_{X^{1, 1/2}_T}^{n-1}:
\]
\begin{align*}
    \eqref{u in Y1 case A ineq}&\lesssim \int_{\Gamma_n}| k_1|\langle \tau_1 - k^3_1\rangle^{-1/2}\widehat{z}_k\prod_{r=1}^n \widehat{u}_{k_r}\,d\Gamma\lesssim \int_{\Gamma_n}\widehat{z}_k\prod_{r=1}^n \widehat{u}_{k_r}\,d\Gamma\\
    &\lesssim \|u\|_{L^2_{x,t}}\|z u^{n-1}\|_{L^2_{x,t}}
    \lesssim \|z\|_{X^{0,1/2-\delta}}\|u\|_{L^2_{x,t}}\|u\|_{X_T^{0, 1/2-\delta}}\|u\|^{n-2}_{X_T^{1/2-\delta, 1/2-\delta}}.
\end{align*}
When $i\ne 1$, we need to show
\[
\eqref{u in Y1 case A ineq}\lesssim T^\varepsilon\|z\|_{X^{0,1/2-\delta}}\|u\|_{X^{0, 1/2}_T}\|u\|_{X^{1, 0}_T}\|u\|_{X^{1, 1/2}_T}^{n-2},
\]
where the estimate follows by the exact same estimate as before, but with $\mathcal{F}^{-1}(\widehat{u}_{k_i})$ being placed in $L^2_{x,t}.$
As before, a small power of $T$ follows by localization in both estimates.

The $\ell^2_kL^1_\tau$ follows exactly the same way as in the original proof of Case A2, but we note that instead of $Y^1_T$ bounds we now have $X^{1, 1/2}_T$ bounds. Indeed, letting $0 < \delta\ll 1$ and applying H\"olders to split off $\langle \tau-k^3\rangle^{1/2+\delta}$ powers reduces the norm we need to estimate to:
\[
\left\| \frac{\langle k\rangle}{\langle \tau - k^3\rangle^{1/2 - \delta}}\mathcal{F}_{x,t}\left( T^n_{\sigma\chi_{H_n\gtrsim k^2}}(u,\cdots, u)\right)\right\|_{L^2_{k,\tau}}.
\]
In exactly the same way we proved A2, we use duality, Plancherel, Condition \eqref{Symmetric Condition for Hn}, $\langle \tau_i-k_i^3\rangle\gtrsim k^2$, and the fact that we have $1/2-\delta$ appearing in every $z$ norm for the proof of \eqref{u in Y1 case A ineq} and conclude the desired result.

\textbf{Case B, C, D}: For these cases, we prove a version of Lemma $\ref{GeneralLemma wo Hn}$ in the specific case that the exponent on $k^*_2k^*_3k^*_4$ is $1/2$ instead of $s_1$. The reason for this will be, as in the prior part of this proof, to gain $X^{1, 1/2}$ bounds for the $\ell^2L^1_\tau$ portion of the proof. Note that even though $s_2 = 1/2$, we have that our bound must be in terms of $X^{1,1/2}$.

Since $\mathcal{NR}$ has no Case $B$ and we've removed single resonances, we see that in all of the remaining cases we have $k^*_2k^*_3k^*_4\gtrsim k^2$ or, better yet, $(k^*_3)^2k^*_4\gtrsim k^2$ in Case D. It follows that:
\[
\frac{|k||\sigma(k, \vec{k})|}{k^*(k^*_2k^*_3k^*_4)^{1/2}} = O(1).
\]
Assuming that we have reorganized the $u$'s in descending order of frequency and that $n\geq 5$, we see that duality reduces the problem to finding a bound on
\begin{equation}\label{u in Y1 BCD int}
\int_{\Gamma_n} \widehat{z}_k|k_1|\widehat{u}_{k_1}|k_2|^{1/2}\widehat{u}_{k_2}|k_3|^{1/2}\widehat{u}_{k_3}|k_4|^{1/2}\widehat{u}_{k_4}\prod_{r=5}^n\widehat{u}_{k_r}\,d\Gamma
\end{equation}
for $z\in X^{0, 1/2}$. Invoking H\"olders, Plancherel, and \eqref{TaoEstiamte} we get
\begin{align*}
    \eqref{u in Y1 BCD int}\lesssim \|z\|_{L^2_{x,t}}&\|(J_xu)(J^{1/2}_xu)^3u^{n-4}\|_{L^2_{x,t}}\\
    &\lesssim\|z\|_{X^{0,0}}\|u\|_{X^{1,0}_T}\|u\|_{X^{1/2, 1/2-\delta}_T}\|u\|_{X_T^{1-\delta,1/2-\delta}}^2\|u\|_{X_T^{1/2-\delta, 1/2-\delta}}^{n-4},
\end{align*}
which yields a small power of $T$ by localization.

The $\ell^2_k L^1_\tau$ norm follows by a similar trick used in most of the prior corollaries. That is, apply H\"olders in time to reduce to bounding the $X^{1, -1/2+\delta}$ norm or, equivalently, bounding \eqref{u in Y1 BCD int} with $z\in X^{0, 1/2-\delta}$. Duality, H\"olders, Plancherel, and \eqref{TaoEstiamte} finishes the proof in exactly the same way as the $X^{1,-1/2}$ estimate done above. Note that we again have a small power of $T$ by localization.
\end{proof}
\begin{remark}\label{REMARK: XSB BOUNDS FOR CASES BCD}
Note that both the $X^{1, -1/2}$ and the $\ell^2_kL^1_\tau$ estimates are bounded by 
\[
T^\varepsilon \left(\|u\|_{X_T^{1,1/2}}^n + \|u\|_{X_T^{1,1/2}}^m\right),
\]
instead of the traditional $Y^1_T$ norm.
\end{remark}
Because the mean removing transformation, $L_t$, ends up changing $f$ from a time independent function into a time and $u$ dependent translation of $f$, we no longer have its disappearance when considering whether or not the integral form of the equation is a contraction. The natural way to obtain a bound of $\|u-v\|_{Y^1_T}$ would then be some kind of a Taylor series and $Y^{1}_T\hookrightarrow L^\infty$ application, but the following Lemma demonstrates how doing so inadvertently requires one more derivative on $f$.
\begin{lemma}\label{forcing Contraction}
Let $f$ be time independent, $s\geq1/2$, $u, v\in X^{s, 1/2}$, $\|u\|_{X^{s,1/2}}, \|v\|_{X^{s,1/2}}\lesssim M$, and $0< T\ll 1$. Then:
\[
\left\|\chi(t/T)\left(L_t[u]f - L_t[v]f\right)\right\|_{Z^s_T}\leq C(M)T^{1/2}\|f\|_{H^{1+s}}\|u-v\|_{X^{0, 1/2}_T}.
\]
\end{lemma}
\begin{proof}
$ $\newline
Consider, first, the $X^{s, -1/2}$ portion:
\begin{align*}
    \bigg\|\frac{\langle k\rangle^s}{\langle \tau - k^3\rangle^{1/2}}\mathcal{F}_{x,t}&\left(\chi(t/T)\left(L_t[u]f - L_t[v]f\right)\right) \bigg\|_{\ell^2_k L^2_\tau}\lesssim \left\|\langle k\rangle^s\mathcal{F}_{x,t}\left(\chi(t/T)\left(L_t[u]f - L_t[v]f\right)\right)\right\|_{\ell^2_kL^2_\tau}\\
    &= \left \|\langle k\rangle^s |\widehat{f}(k)|\left\|\chi(t/T)\left|e^{ik\int_0^t\int_\mathbb{T} anu^{n-1} - anv^{n-1} + bmu^{m-1}-bv^{m-1}\,dx\,ds}-1\right|\right\|_{L^2_t}\right\|_{\ell^2_k}\\
    &\leq \left \|\langle k\rangle^s|k| |\widehat{f}(k)|\left\| \chi(t/T)C(\|u\|_{X^{1/2, 1/2}_T}, \|v\|_{X^{1/2, 1/2}_T})\|u-v\|_{X^{0, 1/2}_T}\right\|_{L^2_{t}}\right\|_{\ell^2_k}\\
    &\lesssim T^{1/2} C(M)\|f\|_{H^{1+s}}\|u-v\|_{X^{0, 1/2}_T},
\end{align*}
where we have used Parsevals in time, the trivial bound $|e^{i\theta} - 1| < |\theta|$, H\"olders in the inner double integral, and (\ref{TaoEstiamte}). 

As for $\ell^2_kL^1_\tau$, we apply H\"olders in the same way that, in Lemma $\ref{GeneralLemma wo Hn}$, we reduced the $\ell^2_k L^1_\tau$ approximation to the $X^{s, -1/3}$ approximation-- that is, the exact same bound considered above, but with $\langle \tau-k^3\rangle ^{-1/3}$ instead of $\langle \tau - k^3\rangle ^{-1/2}$. Trivially estimating out the $\langle \tau - k^3\rangle ^{-1/3}$ term allows us to finish the proof exactly the same way we did following the first inequality above.  
\end{proof}

\begin{lemma}\label{Lemma: Contraction in X(1, 1/2) with X(0,1/2) norm}
Let $n\geq 2$, $0 < T\ll 1$, and $u,v\in X^{1, 1/2}_T$. We then have that
\[
\|\mathcal{R}^2[u] - \mathcal{R}^2[v]\|_{X^{0, -1/2}_T} + \|\mathcal{NR}[u] - \mathcal{NR}[v]\|_{X^{0,-1/2}_T}\leq T^\varepsilon C\left(\|u\|_{X_T^{1, 1/2}}, \|v\|_{X_T^{1, 1/2}}\right)\|u-v\|_{X^{0, 1/2}_T},
\]
where $u = \chi(t/T)u(t)$ and $v =\chi(t/T)v(t)$.
\end{lemma}
\begin{proof}
$ $\newline
We simply proceed by case noting that the above claim will follow if we can prove, under the same hypothesis as the lemma and with $w_2, \cdots, w_{n}\in X^{1, 1/2}$, that
\begin{align*}
\|\mathcal{R}^2[u-v, w_2 &\cdots, w_{n}]\|_{X^{0, -1/2}_T} + \|\mathcal{NR}[u-v, w_{2}, \cdots, w_{n}]\|_{X^{0, -1/2}_T}\\
&\leq T^\varepsilon C\left(\|w_2\|_{X^{1, 1/2}_T}, \cdots, \|w_{n}\|_{X^{1, 1/2}_T}, \|u\|_{X^{1, 1/2}_T}, \|v\|_{X^{1, 1/2}_T} \right)\|u-v\|_{X^{0, 1/2}_T}.
\end{align*}
This, just like every other lemma, follows by considering cases. The crux of the proof in every case is simply that we must only overcome one power of $k$ arising from the symbol $\sigma(k,\vec{k}) = ik$, and in every case but $A$ there is at least one other large frequency  (or combination of other frequencies leading to a large combined frequency). This will allow us great freedom in where we take regularity from to fulfil the single derivative requirement. Note that 
this primarily differs from Lemma \ref{GeneralLemma w Hn} in that we may completely avoid the bad subcase of A2 in which $k_{1}, k_{2}\ne k_{i}.$

Just like in Lemma \ref{GeneralLemma w Hn}, we denote $\widehat{w_j}(\tau_j,k_j)$ by $\widehat{w}_{k_j}$ and proceed by cases:

\textbf{Case A:} This case follows by examining the proofs of Lemmas \ref{GeneralLemma w Hn} and \ref{Lemma: NONLINEAR TERM IN Y1}. Assuming Case A1, the $X^{0,-1/2}$ norm reduces, by duality and Plancherel, to showing
\begin{equation}\label{X0 bound first ineq}
\int_{\Gamma_n}\langle\tau-k^3\rangle^{-1/2}\sigma(k,\vec{k}) \widehat{z}_k(\widehat{u}_{k_1}-\widehat{v}_{k_1})\prod_{r=2}^{n}\widehat{w}_{k_r}\,d\Gamma\lesssim\|z\|_{L_{x,t}^2}\|u-v\|_{X^{0,1/2}_T}\prod_{r=2}^n\|w_r\|_{X^{1,1/2}_T}.
\end{equation}
However, 
\[
\sigma(k,\vec{k})\langle \tau-k^3\rangle^{-1/2} = O(1)
\]
in this regime, so that by Plancherel, H\"olders, and \eqref{TaoEstiamte} we have
\[
\eqref{X0 bound first ineq}\lesssim \|z\|_{L_{x,t}^2}\|u-v\|_{X^{0,1/2-\delta}_T}\prod_{r=2}^n\|w_r\|_{X_T^{1/2-\delta,1/2-\delta}},
\]
which leaves room for a small power of $T$ by localization.

Assuming Case A2, we let $\langle \tau_i - k_i^3\rangle ^{-1/2}$ be the largest of the factors considered in the statement of A2 and additionally assume that $i = 1$. Duality and Plancherel then reduce the problem to showing
\begin{align}\label{X0 bound second ineq}
\int_{\Gamma_n}\langle\tau_i-k_i^3\rangle^{-1/2}\sigma(k,\vec{k}) \widehat{z}_k &(\widehat{u}_{k_1}-\widehat{v}_{k_1})\prod_{r=2}^{n}\widehat{w}_{k_r}\,d\Gamma \\
&\lesssim\|z\|_{X^{0,1/2}}\|u-v\|_{L^2_{x,t}}\prod_{r=2}^n\|w_r\|_{X^{1,1/2}_T}.\nonumber
\end{align}
Since we again have that
\[
\langle \tau_i-k_i\rangle^{-1/2}\sigma(k,\vec{k})=O(1),
\]
it follows by Plancherel, H\"olders, and \eqref{TaoEstiamte} that
\[
\eqref{X0 bound second ineq}\lesssim \|u-v\|_{L^2_{x,t}}\left\|z\prod_{r=2}^nw_r\right\|_{L^2_{x,t}}\lesssim \|u-v\|_{L^{2}_{x,t}}\|z\|_{X^{0,1/2-\delta}}\prod_{r=2}^n\|w_r\|_{X^{1/2-\delta, 1/2-\delta}},
\]
where we again have a small power of $T$ afforded by localization.

If, on the other hand, $i\ne 1$, then we need to show
\[
\eqref{X0 bound second ineq}\lesssim \|z\|_{X^{0,1/2}}\|u-v\|_{X^{0, 1/2}_T}\|w_i\|_{X_T^{1,0}}\prod_{\substack{r=1\\r\ne i}}\|w_r\|_{X^{1,1/2}_T}.
\]
This again reduces by Plancherel, H\"olders, and \eqref{TaoEstiamte} to
\[
\eqref{X0 bound second ineq}\lesssim \|w_i\|_{L^2_{x,t}}\left\|z(u-v)\prod_{\substack{r=2\\r\ne i}}w_r\right\|_{L^2_{x,t}}\lesssim \|z\|_{X^{0,1/2-\delta}}\|u-v\|_{X^{0,1/2-\delta}}\|w_i\|_{L^2_{x,t}}\prod_{\substack{r=1\\r\ne i}}\|w_i\|_{X^{1/2-\delta, 1/2-\delta}},
\]
with again a small power of $T$ by localization.

\textbf{Case B:} In this case, $\mathcal{R}^2$ has three frequencies comparable to $k$ and $\mathcal{NR}$ doesn't contribute by definition. This, then, nearly follows from the proof of Lemma \ref{Lemma: NONLINEAR TERM IN Y1}, since
\[
\frac{|\sigma|}{|k_j|}=O(1)
\]
for any of the three $k_j$ mentioned earlier. For the sake of simplicity, let $j = 2$ (all that matters is that it is not the one of the three possibly associated to $u-v$). It follows by duality and the above fact that we have to prove the bound
\begin{equation}\label{XO Bound Third int}
\int_{\Gamma_n} \widehat{z}_k(\widehat{u}_{k_1}-\widehat{v}_{k_1})\prod_{i=2}^{n}\widehat{w}_{k_i}\,d\Gamma\lesssim \|z\|_{X^{0,1/2}}\|u-v\|_{X^{0,1/2}_T}\|w_2\|_{X^{0,1/2}_T}\prod_{r=3}^n\|w_r\|_{X^{1,1/2}_T}.
\end{equation}
Applying Plancherel, H\"olders, and \eqref{TaoEstiamte} we get
\begin{align*}
\eqref{XO Bound Third int}&\lesssim \|z\|_{L^2_{x,t}}\left\|(u-v)w_2\prod_{r=3}^n w_r\right\|_{L^2_{x,t}}\\
&\lesssim \|z\|_{X^{0,0}}\|u-v\|_{X^{0,1/2-\delta}_T}\|w_2\|_{X^{0,1/2-\delta}_T}\prod_{r=3}^{n}\|w_r\|_{X^{1/2-\delta,1/2-\delta}_T},
\end{align*}
where we again get a small power of $T$ by localization.

\textbf{Case C:} In this case, $k^*_3\gtrsim k$. However, this only means that we have three frequencies comparable to $k$, and hence the proof for Case B handles this.

\textbf{Case D:}
In this case we have $(k^*_3)^2k_4^*\gtrsim k^2$ and again assume that $n\geq 5$. We first split the proof into two cases: $k_1=k^*$ and $k_1\ne k^*$. In the first case, we have 
\[
\frac{|\sigma(k,\vec{k})|}{|k^*_2k^*_3k^*_4|^{1/2}} = O(1),
\]
and hence, without loss of generality, need to prove the bound 
\begin{align}\label{XO bound fourth ineq}
\int_{\Gamma_n} \widehat{z}_k(\widehat{u}_{k_1}-\widehat{v}_{k_1})&\widehat{w}_{k_2}\widehat{w}_{k_3}\widehat{w}_{k_4}\prod_{r=5}^{n}\widehat{w}_{ k_r}\,d\Gamma\\
&\lesssim \|z\|_{X^{0,1/2}}\|u-v\|_{X^{0,1/2}_T}\prod_{r=2}^4\|w_r\|_{X^{1/2, 1/2}_T}\prod_{r=5}^n\|w_r\|_{X^{1,1/2}_T}.\nonumber
\end{align}
Using Plancherel, H\"olders, and \eqref{TaoEstiamte}
\begin{align*}
\eqref{XO bound fourth ineq}&\lesssim \|z\|_{L^2_{x,t}}\left\|(u-v)w_2w_3w_4\prod_{r=5}^nw_r\right\|_{L^2_{x,t}}\\
&\lesssim\|z\|_{X^{0,0}}\|u-v\|_{X^{0,1/2-\delta}_T}\prod_{r=2}^{4}\|w_r\|_{X^{1/2-\delta, 1/2-\delta}_T}\prod_{r=5}^n\|w_r\|_{X^{1/2-\delta, 1/2-\delta}_T}.
\end{align*}
Time localization again gives a small power of $T$.

If, on the other hand, $k_1\ne k^*$ then, without loss of generality, 
\[
\frac{|\sigma(k,\vec{k})|}{|k_2|} = O(1).
\]
It follows by the exact same Plancherel and H\"olders application followed by another (\ref{TaoEstiamte}) used in Case B, that we have the bound
\[
\lesssim \|z\|_{X^{0,0}}\|u-v\|_{X^{0, 1/2-\delta}_T}\prod_{r=2}^{n}\|w_r\|_{X^{1, 1/2-\delta}_T},
\]
with the usual small power of $T$ gain.
\end{proof}

Here we use the prior work and finally proceed to prove local existence for (\eqref{modified LWP PDE}). 
\begin{proof}\textit{Well-posedness of \eqref{modified LWP PDE}}
$ $\newline
Let $u_0\in H^1$, $B := \{ u\in X^{1, 1/2}_T\,:\, \|u\|_{X^{1, 1/2}_T}\leq M\|u_0\|_{H^1}\}$, $0 < T\ll1$, and define the integral operator
\begin{align*}
\Gamma_T[u] := &\chi(t/T)e^{-t\partial_x^3}u_0\\
&+ \int_0^te^{-i(t-s)\partial_x^3}\chi(s/T)\left(L_t[\chi(s/T)u]f - \mathcal{R}^2[\chi(s/T)u]-\mathcal{NR}[\chi(s/T)u]-\gamma u\right)\,ds.
\end{align*}
It follows that a fixed point of $\Gamma_T$ on $B$ would give us a solution of (\ref{modified LWP PDE}) in $X_T^{1, 1/2}$ with the desired bounds. 

Now, by the proof of Lemma \ref{Lemma: NONLINEAR TERM IN Y1}, Remark \ref{REMARK: XSB BOUNDS FOR CASES BCD}, and the proof of Lemma \ref{forcing Contraction}, we see that for any $u\in B$ we have
\[
\|\Gamma_T[u]\|_{X^{1, 1/2}_T}\lesssim \|u_0\|_{H^1} + T^\varepsilon\left(\|f\|_{H^1} + \|u_0\|_{H^1}^n + \|u_0\|_{H^1}^m + \gamma\|u_0\|_{H^1}\right), 
\]
so that for small enough $T = T(\|u_0\|_{H^1}, \|f\|_{H^1}, \gamma)$ and large enough $M$, $\Gamma_T$ maps $B$ to $B$. 

By our prior remarks, $B$ is closed under the $\|\cdot\|_{X^{0,1/2}}$ norm, so it suffices to prove the contraction with this norm. Let $u, v\in B$, and apply Lemmas \ref{forcing Contraction} and \ref{Lemma: Contraction in X(1, 1/2) with X(0,1/2) norm}. It follows immediately that
\[
\|\Gamma_T[u]-\Gamma_T[v]\|_{X^{0, 1/2}_T} \lesssim C(\|u\|_{X^{1, 1/2}_T}) T^\varepsilon \langle\|f\|_{H^1}\rangle\|u-v\|_{X^{0,1/2}_T}.
\]
This proves that $\Gamma_T$ is a contraction on $B$ for a possibly smaller \[T = T(\|u_0\|_{H^1}, \|f\|_{H^1}, \gamma),\] and hence there is a solution, $\tilde{u},$ of (\ref{modified LWP PDE}) in $X^{1, 1/2}.$

It remains to prove the final claim of Proposition \ref{SmootherfTheorem}-- that is, 
\[
\|\tilde{u}\|_{Y^1_T}\leq C(\|f\|_{H^1}, \gamma,
\|u_0\|_{H^1}).\]
This follows by the proof that $\tilde{u}\in Y^1_T$. Indeed, Lemma $\ref{Lemma: NONLINEAR TERM IN Y1}$ demonstrates that the non-linear part of the Duhamel term is in $Y^1_T$, while $f, \gamma\tilde{u},$ and the free solution portions are handled trivially. Furthermore, the conclusion of Lemma \ref{Lemma: NONLINEAR TERM IN Y1} gives $X^{1,1/2}$ bounds on the non-linear term, and hence the bound follows.
\end{proof}
%
%
\section{Smoothing}
Letting $u$ be the solution to $(\ref{SmootherfPDE})$ guaranteed above. We now seek to show that removing a shifted version of the linear solution results in smoothing. That is, we seek to prove the following proposition.
\begin{prop}\label{Smoothing}
Let $u$ a solution (\ref{SmootherfPDE}) at $s = 1$, $0 < \rho < 1$, and $L_t$ defined in $(\ref{L definition})$. There is then a $T =T(\gamma, \|f\|_{H^1}, \|u_0\|_{H^{1}})$ so that \[\|L_t[u]u  - W^\gamma_t u_0\|_{H^{1+\rho}}\leq C\left(\rho, \gamma, \|u_0\|_{H^{1}}, \|f\|_{H^1}\right)\] for all $0\leq t\leq T.$ This persists globally if we have $\textit{a priori}$ control of $\|u\|_{H^1}.$
\end{prop}

\begin{remark}
The lemmas used in both local existence and local smoothing are sufficiently general enough to be applied to $s > 1/2$, as studied in \cite{oh2020smoothing} and sources within. The end point (obtained in \cite{colliander2004multilinear}) should also be obtainable using \eqref{TaoEstiamte} and the ideas present in Section \ref{LWP}. However, we do not pursue this herein since we lack suitable conserved quantities for $1/2 \leq s < 1$.
\end{remark}

\subsection{Smoothing Set-Up}
In this section assume that the non-linearity of $\eqref{SmootherfPDE}$ is again $au^n+bu^m$, with $n\ne m\geq 2$. It will be clear that our methods are unaltered by additional terms.

As previously noted, $\mathcal{R}^2$ must automatically come with at least 3 frequencies comparable to the external frequency, $k$- but the same cannot be said for $\mathcal{NR}$. For example, there could be just one very large frequency with no product of smaller frequencies similar in magnitude to the external frequency. This necessarily lives in Case $A$, in which case (in view of Lemma \ref{GeneralLemma w Hn}) we would have to counteract a $|k|^{1+\rho}$ weight with only a $|k|$ factor, which grants us no smoothing. 

Using this intuition, we do another decomposition in order to remove this poor situation. Indeed, we decompose $\mathcal{NR}[\tilde{u}]$ by defining:
\begin{align*}
    \widehat{\mathcal{HL}}^n[\tilde{u}]_k &:= \sum_{\substack{k_1+\cdots+k_n = k\\ |k_1|\gg \max_{j\ne 1}|k_j|\\ \sum_{j\ne 1}k_j\ne 0\\H_n\gtrsim (k_1)^2}} ik\prod _{j=1}^n \tilde{u}_{k_j}
\end{align*}
and
\begin{align*}
        \widehat{\mathcal{RE}}[\tilde{u}]_k &:= \widehat{\mathcal{NR}}[\tilde{u}]_k - an\widehat{\mathcal{HL}}^n[\tilde{u}]_k - bm\widehat{\mathcal{HL}}^m[\tilde{u}]_k.
\end{align*}
It follows that $\mathcal{RE}$ contains the terms with multiple nearly large internal frequencies as well as the negation of $H_n\gtrsim (k^*)^2$. Note that we have used symmetry to move the highest frequency to the first input, and that $\mathcal{HL}$ carries the Case $A$ restriction, and not simply a large frequency.

We additionally define a normal form transformation by
\begin{equation}\label{NF definition}
T_{\mathcal{NF}}^n[u_1, \cdots, u_n] := \sum_{\substack{k_1+\cdots+k_n = k\\ k_{1}\gg \max{k_2,\cdots, k_n}\\k_2+\cdots +k_n\ne 0\\ H_n\gtrsim k_1^2}}\frac{k}{H_n(k_1,\cdots,k_n)}e^{ikx}\prod_{j=1}^n \widehat{u}_{k_j}.
\end{equation}

\begin{remark}
Normal form transformations have a long history in the periodic setting (Recently, see: \cite{erdogan2017smoothing, erdougan2013global,oh2013resonant}) and, in this case, the stipulation that $H_n\gtrsim (k^*)^2$ makes the connection to other approaches to smoothing very clear (see \cite{correia2020nonlinear}). Indeed, if we perform a change of variables on the initial PDE \eqref{AttractorPDEoNR}  with $\mathbb{K} = \mathbb{T}$ and non-linearity $g(x) = x^{k}$ and set $u = W_tv$, then $v$ satisfies an equation of the form
\[
\widehat{v}(k, t) = \widehat{v}(k,0) - \int_0^t \sum_{\Omega}ike^{-is(k^3-k_1^3-\cdots- k_{n}^3)}\prod_{i=1}^{n}\widehat{u}(k_i, s)\,ds,
\]
where the exponent is $-isH_n$. It follows that when $H_n\gtrsim (k^*)^2 > 0$ we can perform integration by parts and pick up a large denominator-- an essential observation.
\end{remark}
The following lemma places a bound on the amount of smoothing we will be able to get:
\begin{lemma}
It follows from this definition that for any $s\in\mathbb{R}$ and $\varepsilon > 0$, we have
\[
\|T_\mathcal{NF}^{n}[u_1(x,0), u_{2}(x,0), \cdots, u_n(x,0)]\|_{H^{s+1}}\lesssim_\varepsilon \|u_1(x,0)\|_{H^s_x}\prod_{r=2}^n\|u_r(x,0)\|_{H^{1/2+\varepsilon}}.
\]
\end{lemma}
\begin{proof}
$ $\newline
Since we are, by definition, in Case A, we must have that $\sigma(k,\vec{k}) = O(1/k)$. Denoting $\widehat{u_i}(k_i)$ by $\widehat{u}_{k_i}$, the result follows fairly quickly by H\"olders and Sobolev embedding:
\[
\lesssim \|J^s_x u_1(x,0)\|_{L^2}\prod_{r> 1}\|u_r\|_{L^\infty}\lesssim\|u_1\|_{H^s}\prod_{r > 1}\|u_r\|_{H^{1/2+\varepsilon}}.
\]
\end{proof}
Using \eqref{NF definition} and symmetry, we see that
\begin{align*}
\left(\partial_t-ik^3+\gamma\right)&\mathcal{F}_x\left(T_{\mathcal{NF}}^{n}\right)(W^\gamma_t u_0, \tilde{u}, \cdots, \tilde{u})_k\\
&= \sum_{\substack{k_1+\cdots+k_n = k\\ k_{1}\gg \max{|k_2|,\cdots, |k_n|}\\k_2+\cdots +k_n\ne 0\\H_n\gtrsim k_1^2}}\frac{k(ik_1^3-\gamma)}{H_n(k_1,\cdots, k_n)}W^\gamma_tu_0\prod_{r=2}^n \tilde{u}_{k_r}\\
&\qquad+\sum_{\substack{k_1+\cdots+k_n = k\\ k_{1}\gg \max{|k_2|,\cdots, |k_n|}\\k_2+\cdots +k_n\ne 0\\H_n\gtrsim k_1^2}}\frac{k(n-1)}{H_n(k_1,\cdots, k_n)}W^\gamma_tu_0\partial_tv_{k_r}\prod_{r=3}^{n} \tilde{u}_{k_r} \\
&\qquad+\sum_{\substack{k_1+\cdots+k_n = k\\ k_{1}\gg \max{|k_2|,\cdots, |k_n|}\\k_2+\cdots +k_n\ne 0\\H_n\gtrsim k_1^2}}\frac{k(-ik^3+\gamma)}{H_n(k_1,\cdots, k_n)}W^\gamma_tu_0\prod_{r=2}^n \tilde{u}_{k_r}\\
&=-\sum_{\substack{k_1+\cdots+k_n = k\\ k_{1}\gg \max{|k_2|,\cdots, |k_n|}\\k_2+\cdots +k_n\ne 0\\H_n\gtrsim k_1^2}}ikW^\gamma_tu_0\prod_{r=2}^n \tilde{u}_{k_r}\\
&\qquad +\sum_{\substack{k_1+\cdots+k_n = k\\ k_{1}\gg \max{|k_2|,\cdots, |k_n|}\\k_2+\cdots +k_n\ne 0\\H_n\gtrsim k_1^2}}\frac{k(n-1)}{H_n(k_1,\cdots, k_n)}W^\gamma_tu_0(\partial_t-ik_{2}^3) \tilde{u}_{k_2}\prod_{r=3}^{n} \tilde{u}_{k_r}\\
&=-\mathcal{F}_x\left(\mathcal{HL}^{n}\right)[W_t^\gamma u_0, \tilde{u}, \cdots, \tilde{u}]_k\\
&\qquad+ (n-1)\mathcal{F}_x\left(T_\mathcal{NF}^{n}\right)[W_t^\gamma u_0, (\partial_t +\partial_x^3)\tilde{u}, \tilde{u}, \cdots, \tilde{u}]_k.
\end{align*}

The above calculation is the reasoning behind the definition of $\mathcal{HL}$, and motivates the following definition:
\begin{align}\label{v smooth definition}
\tilde{u} &= W_t^\gamma u_0 + an T_{\mathcal{NF}}^n[W_t^\gamma u_0, \tilde{u}, \cdots, \tilde{u}] + bmT_{\mathcal{NF}}^m[W_t^\gamma u_0, \tilde{u}, \cdots, \tilde{u}] + \mathcal{F}_x^{-1}(L_t[u]\widehat{F}) + v,
\end{align}
where
\[
\widehat{F} = \frac{\widehat{f}}{\gamma-ik^3},
\]
and 
\begin{align*}
(\partial_t-ik^3+\gamma)L_t[u]\widehat{F} &= L_t[u]\widehat{f} + L_t[u]ik\widehat{F}\int_\mathbb{T} an\tilde{u}^{n-1} + bm\tilde{u}^{m-1}\,ds\\
&:=\mathcal{F}_x\left(\tilde{f} + F_x G\right),
\end{align*}
where
\[
\widehat{G} = L_t[u]\int_\mathbb{T} an\tilde{u}^{n-1} + bm\tilde{u}^{m-1}\,ds.
\]

\begin{remark}
The definition of $v$ is reminiscent of the result in  \cite{wang2015global}, which asserted the existence of a global attractor for the 4-gKdV in $H^s(\mathbb{R})$ with forcing $f\in L^2\cap H^{s-3}$. A trick similar in nature to the $u-v$ one used above allows one to work with a modified equation that has sufficiently smooth forcing, at the cost of smooth potential terms.
\end{remark}

Substitution into the equation yields that $v$ satisfies the equation:
\begin{align*}
    v_t&+v_{xxx}+\gamma v =-F_xG-\mathcal{R}^2[\tilde{u}] - \mathcal{RE}[\tilde{u}]-\mathcal{HH}[\tilde{u}]\\
                         & - an(n-1)T_{\mathcal{NF}}^n[W^\gamma_t u_0, \mathcal{NR}[\tilde{u}]+\mathcal{R}^2[\tilde{u}] +\gamma \tilde{u} + \tilde{f}, \tilde{u}, \cdots, \tilde{u}]\\
                         &- bm(m-1) T_{\mathcal{NF}}^m[W^\gamma_t u_0, \mathcal{NR}[\tilde{u}]+\mathcal{R}^2[\tilde{u}] +\gamma \tilde{u} + \tilde{f}, \tilde{u}, \cdots, \tilde{u}]\\
                         & - an\mathcal{HL}^n[an T_{\mathcal{NF}}^n[W_t^\gamma u_0, \tilde{u}, \cdots, \tilde{u}] + bmT_{\mathcal{NF}}^m[W_t^\gamma u_0, \tilde{u}, \cdots, \tilde{u}]\\
                         &\qquad\qquad+ \mathcal{F}_x^{-1}(L_t[u]\widehat{F}) + v, \tilde{u}, \cdots, \tilde{u}]\\
                         & - bm\mathcal{HL}^m[an T_{\mathcal{NF}}^n[W_t^\gamma u_0, \tilde{u}, \cdots, \tilde{u}] + bmT_{\mathcal{NF}}^m[W_t^\gamma u_0, \tilde{u}, \cdots, \tilde{u}]\\
                         &\qquad\qquad+ \mathcal{F}_x^{-1}(L_t[u]\widehat{F}) + v, \tilde{u}, \cdots, \tilde{u}],
\end{align*}
with initial condition 
\begin{align*}
v(x,0) = -an T^n_\mathcal{NF}[u_0, \cdots, u_0] -bm T_\mathcal{NF}^m[u_0,\cdots, u_0] - F,\nonumber
\end{align*}
Observe that the above equation is obtained by replacing the first input of all of the $\mathcal{HL}$'s with \eqref{v smooth definition}.

\subsection{Smoothing Lemmas}
As in the prior lemmas, we assume that $\tilde{u} = \chi(t/T)\tilde{u}(x,t)$, $v = \chi(t/T)v(x,t)$, $F = \chi(t/T)F(x,t)$, and $u_0^\gamma= \chi(t/T)W_t^\gamma u(x,0)$ where $\chi$ is a smooth cutoff function that is $1$ on $[-1, 1]$ and $0$ outside $[-2, 2]$, and $n\ne m\geq 2$.

\begin{remark}
As a reminder, for $n\geq 4$ we defined in Lemma \ref{GeneralLemma w Hn} the cases

\begin{itemize}
            \item[A.] $H_n\gtrsim (k^*)^2$.
            \item[B.] $k_{j_0} = k$ for some $j_0\in\{1, \cdots, n\}.$
            \item[C.] ${k^*_3}\gtrsim k$
            \item[D.] $(k^*_3)^2k^*_4\gtrsim (k^*)^2$.
\end{itemize}
This notation will be frequently used in the lemmas to come.
\end{remark}
\begin{lemma}
If $\rho < 1$, then there is $\varepsilon > 0$ so that for $0 < T\ll 1$:
\begin{align*}
\|\mathcal{HH}[\tilde{u}]\|_{Z^{1+\rho}_T}\lesssim_\varepsilon T^\varepsilon\left(\tilde{u}\|_{Y^1_T}^n+\|\tilde{u}\|_{Y^1_T}^m\right).
\end{align*}
\end{lemma}
\begin{proof}
$ $\newline
In this regime, we must have that there are at least two high internal frequencies and we have no Case $B$. Also note that $\sigma(k,\vec{k}) = ik$. Proceeding by cases:

\textbf{Case A}: We have $H_n\gtrsim (k^*)^2$, so it follows by condition $\eqref{Symmetric Condition for Hn}$ of Lemma \ref{GeneralLemma w Hn} with $s_0 = 1+\rho$ and $s_1 =  1$ that 
\[
\frac{|k|^{\rho}|ik|}{|k|^{2}} = |k|^{\rho - 1 } = O(1),
\]
for $\rho < 1$.

\textbf{Cases C, D:} We have that $(k^*_2k^*_3k^*_4)\gtrsim (k^*)^2, $ so condition $\eqref{Symmetric Condition wo Hn}$ of Lemma \ref{GeneralLemma wo Hn} with $s_0 = 1+\rho$ and $s_1 = 1$ gives:
\[
\frac{|k|^{1+\rho}|ik|}{|k^*|^{1-\varepsilon}|k^*_2k^*_2k^*_3|}\lesssim  \frac{|k|^{2+\rho}}{|k|^{3-\varepsilon}} \sim |k|^{\rho - 1+\varepsilon} = O(1).
\]
\end{proof}

\begin{lemma}
If $\rho < 1$, there is $\varepsilon > 0$ so that for $ 0 <T\ll 1$:
\begin{align*}
\|\mathcal{R}^2[\tilde{u}]\|_{Z^{1+\rho}_T}\lesssim_\varepsilon T^\varepsilon \left(\|\tilde{u}\|_{Y^1_T}^{n} + \|\tilde{u}\|^m_{Y^1_T}\right).
\end{align*}
\end{lemma}
\begin{proof}
$ $\newline
Note that $\sigma(k,\vec{k})= ik$. Proceeding by cases:

\textbf{Case A:} We have $H_n\gtrsim (k^*)^2$, so using the fact that we must have at least two high internal frequencies (larger than  $k$), we use condition $\eqref{Symmetric Condition for Hn}$ of Lemma \ref{GeneralLemma w Hn} with $s_0 = 1+\rho$, $s_1 = 1$ to get
\[
\frac{|k|^{\rho}|ik|}{|k|^{2}}\lesssim \frac{|k|^{2+\rho}}{|k|^{2}} = |k|^{\rho - 1} = O(1),
\]
which gives the desired $\rho \leq 1$.

\textbf{Cases B, C, D:} In all of these cases we necessarily have that $|k^*_2k^*_3k^*_4|\gtrsim k^2$, so that condition $\eqref{Symmetric Condition wo Hn}$ of Lemma \ref{GeneralLemma wo Hn} with $s_0 = 1+\rho$ and $s_1 = 1$ gives
\[
\frac{|k|^{1+\rho}|ik|}{|k^*|^{1-\varepsilon}|k^*_2k^*_2k^*_3|} \lesssim |k|^{\rho + \varepsilon - 1} = O(1),
\]
which is good for $\rho < 1$.
\end{proof}
\begin{lemma}
Let $\rho < 1$. There is then an $\varepsilon > 0$ so that for $0 < T\ll 1$:
\[
\|\mathcal{RE}[\tilde{u}]\|_{Z^{1+\rho}_T}\lesssim_\varepsilon T^\varepsilon \left(\|\tilde{u}\|_{Y^1_T}^n + \|\tilde{u}\|_{Y^1_T}^m\right).
\]
\end{lemma}
\begin{proof}
$ $\newline
The symbol is $\sigma(k,\vec{k}) = ik$. Considering only the degree $n$ non-linearity first, we have the restriction that $H_n\ll (k^*)^2$ and that, without loss of generality, $k_1\gg \max(|k_2|,\cdots, |k_n|).$ It follows that we must have Case D, and hence, by condition $\eqref{Symmetric Condition wo Hn}$ of Lemma \ref{GeneralLemma wo Hn} with $s_0 = 1+\rho$ and $s_1 = 1$:
\[
\frac{|k|^{1+\rho}|ik|}{|k^*|^{1-\varepsilon}|k^*_2k^*_2k^*_3|}\lesssim k^{\rho +\varepsilon - 1} = O(1),
\]
for $\rho < 1$.
\end{proof}

\begin{lemma}
If $\rho < 1$, then there exists $\varepsilon > 0$ such that for any $0 < T\ll 1$:
\[
\|T^{n}_{\mathcal{NF}}[u_0^\gamma,\tilde{f}, \tilde{u}, \cdots, \tilde{u}]\|_{Z^{1+\rho}_T}\lesssim T^\varepsilon\|u_0\|_{H^1}\|f\|_{H^{1}}\|\tilde{u}\|_{C^0_tH^1_x}^{n-2}.
\]
\end{lemma}
\begin{proof}
$ $\newline
Here our 
\[
\sigma(k,\vec{k}) = \frac{k}{H_n(k_1, \cdots, k_n)},
\]
and by construction we must have that $H_n(k_1, \cdots, k_n)\gtrsim (k_1^*)^2$, so that $\sigma(k,\vec{k}) = O(\frac{1}{k}).$ It follows that
\begin{align*}
    \|T^{n}_{\mathcal{NF}}[u_0^\gamma,\tilde{f}, \tilde{u}, \cdots, \tilde{u}]\|_{Z^{1+\rho}_T}&\lesssim T^{1/2}\|T^{n}_{\mathcal{NF}}[u_0^\gamma,\tilde{f}, \tilde{u}, \cdots, \tilde{u}]\|_{X^{s+\rho,0}}\\
    &\lesssim T^{1/2}\|u_0^\gamma\|_{X^{s+\rho-1,0}}\|\tilde{f}\|_{L^{\infty}_{x,t}}\|\tilde{u}\|_{L^{\infty}_{x,t}}^{n-2}.
\end{align*}
The result then follows by Sobolev embedding for $\rho\leq 1$.
\end{proof}

\begin{lemma}
If $\rho < 1$, then there exists $\varepsilon > 0$ such that for any $0 < T\ll 1$:
\begin{align*}
\|T^{n}_{\mathcal{NF}}[u_0^\gamma,\left(\mathcal{R}^2 + \mathcal{NR}\right)[\tilde{u}], \tilde{u}, \cdots, \tilde{u}]\|_{Z^{1+\rho}_T}\lesssim T^\varepsilon\|u_0\|_{H^1}\left(\|\tilde{u}\|_{Y^1_T}^{2n-2} + \|\tilde{u}\|_{Y^1_T}^{n+m-2}\right).
\end{align*}
\end{lemma}
\begin{proof}
$ $\newline
We assume that $\mathcal{R}^2$ and $\mathcal{NR}$ contain only $\tilde{u}^m$ as a non-linearity, and let $s_0 = 1+\rho $ and $s_1 = 1$. We then have
\[
\sigma(k,\vec{k}) = \frac{k(k_2+\cdots+k_{m+1})}{H_n(k_1, k_2+\cdots+k_{m+1}, k_{m+2}, \cdots, k_{m+n-1})},
\]
where $|k_1| \gg \max\{|k_2 + \cdots + k_{m+1}|, |k_{m+2}|, \cdots, |k_{n+m-1}|\}$ and the denominator is $\gtrsim k_1^2.$ It follows from this that $\sigma(k,\vec{k}) = O(1).$

Now, proceeding case by case:

\textbf{Case A:} Here, we must have $H_{n+m-1}\gtrsim (k^*)^2$ and hence Lemma \ref{GeneralLemma w Hn} condition \eqref{Symmetric Condition for Hn} becomes
\[
\frac{|k|^{\rho}}{|k^*|}\lesssim |k|^{\rho-1} = O(1),
\]
for $\rho < 1$.

\textbf{Case B:} To be in Case $B$, we must either have $k_1 = k$ or $k = k_j$ for some $j\in \{2, \cdots, m+1\}.$ In the first case we must have that $k_j\gtrsim k_{2}+\cdots+k_{m+1}$ for some $j\not\in\{1, 2, \cdots, m+1\}$. Hence by the frequency restrictions on $T_\mathcal{NF}^{n}$ with $s_0 = 1+\rho, s_1= 1$, \eqref{Symmetric Condition wo Hn} becomes:
\[
\frac{|k|^{1+\rho}k(k_s+\cdots+k_{s+m-1})}{|k^*|^{1-\varepsilon}|k^*_2k_3^*k^*_4|\left(H_n(k_1, k_2+\cdots+k_{m+1}, k_{m+2}\cdots, k_{m+n-1})\right)} \lesssim \frac{|k|^{2+\rho}}{|k|^{3-\varepsilon}} \lesssim |k_1|^{\rho - 1 + \varepsilon} = O(1)
\]
if $\rho < 1$. This case then follows by Lemma \ref{GeneralLemma wo Hn}.

On the other hand, if we have a frequency in the $\{2, \cdots, m+1\}$ range that is equal to $k$, we must have another large frequency similar to $|k_1|$. We thus have three frequencies on the order of $|k|$ and another application of Lemma \ref{GeneralLemma wo Hn} gives:
\[
\frac{k^{1+\rho}}{|k^*|^{1+\varepsilon}|k^*_2k^*_3k^*_4|}\lesssim \frac{k^{1+\rho}}{|k|^{3+\varepsilon}}\lesssim |k|^{\rho + \varepsilon - 2} = O(1)
\]
for $\rho < 2$. 

\textbf{Cases C, D:} In both of these cases we must have $|k^*_2k^*_3k^*_4|\gtrsim |k^*|^2$, so we must have, by Lemma \ref{GeneralLemma wo Hn} and condition \eqref{Symmetric Condition wo Hn}:
\[
\frac{k^{1+\rho}}{|k^*|^{1+\varepsilon}|k^*_2k^*_3k^*_4|}\lesssim \frac{k^{1+\rho}}{|k|^{3+\varepsilon}}\lesssim |k|^{\rho + \varepsilon - 2} = O(1),
\]
for $\rho < 2.$ 
\end{proof}

\begin{lemma}
Let $\rho < 1$. Then there exists $\varepsilon > 0$ such that for any $0 < T\ll 1:$
\[
\|\mathcal{HL}^{n}[T_\mathcal{NF}^{m}[v_0^\gamma, v, w],v, w]\|_{Z^{1+\rho}_T}\lesssim_\varepsilon T^\varepsilon \|v_0\|_{H^1}\|v\|_{Y^1_T}^{\ell_1+\ell_2-2}.
\]
\end{lemma}
\begin{proof}
$ $\newline
The symbol for this estimate is 
\[
\sigma = \frac{k(k_1+\cdots + k_m)}{H_m(k_1, \cdots, k_m)},
\]
where we have the restrictions
\begin{align}
k_1&\gg\max\{|k_2|, \cdots, |k_m|\},\\
k_1+\cdots+k_m&\gg \max\{|k_{m+1}|, \cdots, |k_{m+n-1}|\},\\
(k_1+\cdots k_m)^2 &\lesssim H_n(k_1+\cdots +k_m, k_{m+1}, \cdots, k_{m+n-1}),\mbox{ and}\\
k_1^2&\lesssim H_m(k_1,\cdots, k_m).
\end{align}
From the above, it follows that 
\begin{align}
    k_1^2 &\lesssim H_n(k_1+\cdots +k_m, k_{m+1}, \cdots, k_{m+n-1})\mbox{ and}\\
    k_1&\gg\max\{|k_2|, \cdots, |k_{n+m-1}|\}. 
\end{align}
By adding together the restrictions on $H_n$ and $H_m$ we find
\begin{align*}
k_1^2&\lesssim H_{n+m-1}(k_1, \cdots, k_{n+m-1})\\
\sigma &= O(1).
\end{align*}
As $H_{m+n-1}\gtrsim k_1^2$, this enables us to apply Lemma 2 Condition \eqref{Symmetric Condition for Hn} with $s_0 = s+\rho$ and $s_1 = 1$, so that we have:
\[
\frac{|k|^{\rho}}{|k^*|}\lesssim \frac{|k|^{\rho}}{|k|}= |k|^{\rho - 1} = O(1),
\]
for $\rho \leq 1$.
\end{proof}
\begin{remark}
In the prior lemma we concluded that $H_{m+n-1}\gtrsim k_1^2$, which does not preclude the possibility of resonance. However, Lemma \ref{GeneralLemma w Hn} is insensitive to this, allowing us to avoid appealing to cancellation properties as in \cite{oh2020smoothing}.
\end{remark}
\begin{lemma}
If $\rho \in\mathbb{R}$, then there exists $\varepsilon > 0$ such that for $0 < T\ll 1$:
\[
\|\mathcal{HL}^{n}[v, \tilde{u}, \cdots, \tilde{u}]\|_{Z^{1+\rho}_T} \lesssim_\varepsilon T^{\varepsilon}\|v\|_{Y^{1+\rho}_T}\|\tilde{u}\|_{Y^1_T}^{n-1}.
\]
\end{lemma}
\begin{proof}
$ $\newline
Here, our $\sigma(k,\vec{k}) = ik$ and by definition we satisfy the condition of Case A. Taking $s_0 = 1+\rho$ and $s_1 = 1$, we apply Corollary \ref{Large Term Estiamte} with Condition \eqref{Symmetric Condition for Hn} taking the form:
\[
\frac{|k|^{\rho}|ik|}{|k^*|^{1+\rho}} = O(1).
\]

\end{proof}

\begin{proof}\textit{Proposition \ref{Smoothing}}
$ $\newline
Let $T$ be chosen to satisfy Proposition $\ref{SmootherfTheorem}$. Using Duhamel and applying the lemmas and remarks in this section we conclude the result. Specifically, we find
\begin{align*}
\|v\|_{Y_T^{1+\rho}}&\lesssim_{\gamma, \varepsilon, \rho} T^\varepsilon\|v\| _{Z^{1+\rho}_T}C_1(\|f\|_{H^1_x}, \|\tilde{u}\|_{Y^1_T}) + C_2(\|f\|_{H^1_x}, \|\tilde{u}\|_{Y^1_T})\\
&\qquad\qquad+ \|F\|_{H^{1+\rho}_x}\left(\|\tilde{u}\|_{Y^1_T}^{n-1} + \|\tilde{u}\|_{Y^1_T}^{m-1}\right)\\
&\lesssim T^\varepsilon\|v\| _{Z^{1+\rho}_T}C_1(\|f\|_{H^1_x}, \|\tilde{u}\|_{Y^1_T}) + C_3(\|f\|_{H^1_x}, \|\tilde{u}\|_{Y^1_T}).
\end{align*}

We then choose $T = T(\|f\|_{H^1}, \rho, \gamma, \|u_0\|_{H^1})$ smaller than the local well-posedness size and small enough so that $T^\varepsilon C_1(\|f\|_{H^1_x}, \|\tilde{u}\|_{Y^1_T})\ll 1$, and apply the triangle inequality and the inequality at the end of Proposition $\ref{SmootherfTheorem}$, to conclude
\[
\|v\|_{Y^{1+\rho}_T} \leq C(\|f\|_{H^1}, \rho, \gamma, \|u_0\|_{H^1}),
\]
for $0 < t < T$.
Using the definition of $z$ given at $(\ref{v smooth definition})$, the triangle inequality, and the $C^0_tH^{1+\rho}_x$ embedding, we obtain:
\begin{align*}
\|\tilde{u} - W_t^\gamma u(x,0)\|_{C^0_tH^{1+\rho}_x}&\lesssim \|v\|_{Y^{1+\rho}_T} + \|T_\mathcal{NF}^n[\tilde{u}, \cdots, \tilde{u}]\|_{C^0_tH^{1+\rho}_x}\\
&\qquad\qquad+ \|T_\mathcal{NF}^m[\tilde{u}, \cdots, \tilde{u}]\|_{C^0_tH^{1+\rho}_x} + \|f\|_{H^1}\\
&\lesssim C(\|f\|_{H^1}, \rho, \gamma, \|u_0\|_{H^1}),
\end{align*}
as desired.

When we have control of the $H^1$ norm of $u_0$, then this extends globally by noting that the above arguments imply (by iteration) that for $t\in[NT, (N+1)T)$ and $\rho < 1:$
\[
\|\tilde{u}(x,t)-W_{t-NT}^\gamma \tilde{u}(x,NT)\|_{H^{1+\rho}_x}\lesssim C(\|f\|_{H^1},\rho,\gamma, \|u_0\|_{H^1}).
\]
We then apply a simple telescoping argument from, e.g., Theorem 4.3 of \cite{erdougan2016dispersive}.
\end{proof}

%
%

\section*{Proof of Theorem 1}

\begin{proof}[Proof of Theorem \ref{Theorem3}]
$ $\newline
We follow a general strategy presented in \cite{erdogan2011long} for proving the existence of a global attractor given a smoothing statement akin to Proposition \ref{Smoothing} and suitable compact embeddings. Let $u_0$ be mean zero. In light of Lemma \ref{AbsorbingSet}, we see that solutions to $(\ref{AttractorPDEoNR})$ with $g$ described in Theorem \ref{Theorem3} can be extended globally in time. It follows that the results of Proposition \ref{Smoothing} can, as well. 

Now, let $0 < \rho < 1,$ $S_t$ be the data-to-solution map at time $t$, $B$ the absorbing set (guaranteed by Remark $\ref{AbsorbingSet}$), $u_0\in B$, and $S_tu_0 = L_{t}^{-1}[S_tu_0]W^\gamma_tu_0+ N_tu_0$. It's immediate that, uniformly on $B$, \begin{equation}\label{uniformDecay}
    \|L_{t}^{-1}[S_tu_0]W^\gamma_tu_0\|_{H^1} = \|W^\gamma_t u_0\|_{H^1} \lesssim e^{-\gamma t}\to 0.
\end{equation} Furthermore, by Proposition \ref{Smoothing} and Lemma \ref{AbsorbingSet} we see that 
\begin{align*}
    \|N_tu_0\|_{H^{1+\rho}} & = \left\|\left(L_{t}[S_tu_0]S_t - W_t^\gamma\right) u_0\right\|_{H^{1+\rho}}\\
    &\leq C(\gamma, \rho, \|f\|_{H^{1}}),
\end{align*}
and hence, by Rellich's Theorem, $\{N_tu_0\,:\,t > 0\}$ is pre-compact in $H^{1}$. It follows that $S_t$ is asymptotically compact, ensuring the existence of a global attractor, A (For more information, see \cite{temam2012infinite}).
To prove the final claim we let $0 < \rho < 1$, note that \[A = \omega(B) = \bigcap_{\tau} \overline{\bigcup_{t > \tau} S_t B}:= \bigcap_\tau U_\tau,\] and define $B_\rho$ to be the ball of radius $C(\gamma, \rho, \|f\|_{H^{1}})$ in $H^{\rho/2+3/2}$. By our proof of the existence of an absorbing set, there is a time $T = T(\gamma, \|u_0\|_{H^1}, \|f\|_{H^1})$ so that for $t > T$,
\begin{align*}
\|S_tu_0 - W_{t-T}^\gamma L_{t}^{-1}[S_tu_0]L_T[S_tu_0]S_Tu_0\|_{H^{1+\rho}(\mathbb{T})}&= \|\tilde{u}(t) - W_{t-T}^\gamma \tilde{u}(T)\|_{H^{1+\rho}(\mathbb{T})}\\
&\leq C(\rho, \gamma, \|f\|_{H^1(\mathbb{T})}),
\end{align*}
with 
\begin{align*}
\|W_{t-T}^\gamma L_{t}^{-1}[S_tu_0]&L_T[S_tu_0]S_Tu_0\|_{H^1(\mathbb{T})} = e^{-(t-T)\gamma}\|S_Tu_0\|_{H^1(\mathbb{T})}\\
&\leq e^{-(t-T)\gamma}C(\gamma, \|f\|_{H^{1}(\mathbb{T})})\overset{t\to\infty}{\to} 0.
\end{align*}
Now, by Proposition $\ref{Smoothing}$ and iteration, at any time $t > T$ guaranteed by the prior sentence, we see that \[\|u\|_{H^{1+\rho}}\lesssim C(\gamma, \rho, \|f\|_{H^1}).\] Hence, for $\tau > T$, $U_\tau \subset B_\rho + B(\delta_\tau, 0)$ in $H^1$, where $\delta_\tau \to 0$ as $\tau\to \infty$. Since $B_\rho$ is compact in $H^1$, it follows that $A\subset B_\rho$, a compact set in $H^{1+\rho}$ by Rellich's theorem.
\end{proof}

\subsection{Acknowledgements}
The author would like to thank Professor Burak Erdo{\u{g}}an for not only recommending this problem and participating in many helpful discussions, but also for putting up with my \textit{terrible} questions. His endless patience was a crucial ingredient.
\bibliographystyle{abbrv}
\bibliography{Attractor.bib, gKdV.bib, GeneralDispersivePDE.bib, NLS.bib}

\begin{thebibliography}{10}

\bibitem{bao2013global}
J.~Bao and Y.~Wu.
\newblock Global well-posedness for periodic generalized {K}orteweg-de {V}ries
  equation.
\newblock {\em Indiana Univ. Math. J.}, 66:1797--1825, 2017.

\bibitem{bourgain1993fourier}
J.~Bourgain.
\newblock Fourier transform restriction phenomena for certain lattice subsets
  and applications to nonlinear evolution equations.
\newblock {\em Geometric \& Functional Analysis GAFA}, 3(3):209--262, 1993.

\bibitem{colliander2001global}
J.~Colliander, M.~Keel, G.~Staffilani, H.~Takaoka, and T.~Tao.
\newblock Global well-posedness for schr{\"o}dinger equations with derivative.
\newblock {\em SIAM Journal on Mathematical Analysis}, 33(3):649--669, 2001.

\bibitem{colliander2003sharp}
J.~Colliander, M.~Keel, G.~Staffilani, H.~Takaoka, and T.~Tao.
\newblock Sharp global well-posedness for kdv and modified {K}d{V} on
  $\mathbb{T}$ and $\mathbb{R}$.
\newblock {\em J. Amer. Math. Soc}, 16:705, 2003.

\bibitem{colliander2004multilinear}
J.~Colliander, M.~Keel, G.~Staffilani, H.~Takaoka, and T.~Tao.
\newblock Multilinear estimates for periodic {K}d{V} equations, and
  applications.
\newblock {\em Journal of Functional Analysis}, 211(1):173--218, 2004.

\bibitem{colliander2007resonant}
J.~Colliander, M.~Keel, G.~Staffilani, H.~Takaoka, and T.~Tao.
\newblock Resonant decompositions and the $ {I} $-method for the cubic
  nonlinear {S}chr{\"o}dinger equation on $\mathbb{R}^2$.
\newblock {\em Discrete \& Continuous Dynamical Systems-A}, 21(3):665, 2008.

\bibitem{correia2020nonlinear}
S.~Correia and J.~D. Silva.
\newblock Nonlinear smoothing for dispersive pde: a unified approach.
\newblock {\em Journal of Differential Equations}, 2020.

\bibitem{DLOTKO20093934}
T.~Dlotko, M.~B. Kania, and M.~Yang.
\newblock Generalized {K}orteweg–de {V}ries equation in ${H}^1(\mathbb{R})$.
\newblock {\em Nonlinear Analysis: Theory, Methods \& Applications}, 71(9):3934
  -- 3947, 2009.

\bibitem{erdogan2017smoothing}
M.~Erdogan, T.~Gurel, and N.~Tzirakis.
\newblock Smoothing for the fractional {S}chr{\"o}dinger equation on the torus
  and the real line.
\newblock {\em Indiana Univ. Math. J.}, 68:369--392, 2019.

\bibitem{erdougan2013global}
M.~B. Erdo{\u{g}}an and N.~Tzirakis.
\newblock Global smoothing for the periodic {K}d{V} evolution.
\newblock {\em International Mathematics Research Notices},
  2013(20):4589--4614, 2013.

\bibitem{erdogan2011long}
M.~B. Erdo{\u{g}}an and N.~Tzirakis.
\newblock Long time dynamics for forced and weakly damped {K}d{V} on the torus.
\newblock {\em Communications on Pure \& Applied Analysis}, 12(6), 2013.

\bibitem{erdougan2016dispersive}
M.~B. Erdo{\u{g}}an and N.~Tzirakis.
\newblock {\em Dispersive partial differential equations: wellposedness and
  applications}, volume~86.
\newblock Cambridge University Press, 2016.

\bibitem{farah2010supercritical}
L.~G. Farah, F.~Linares, and A.~Pastor.
\newblock The supercritical generalized {K}d{V} equation: Global well-posedness
  in the energy space and below.
\newblock {\em arXiv preprint arXiv:1009.3234}, 2010.

\bibitem{goubet2000asymptotic}
O.~Goubet.
\newblock Asymptotic smoothing effect for weakly damped forced {K}orteweg-de
  {V}ries equations.
\newblock {\em Discrete \& Continuous Dynamical Systems-A}, 6(3):625, 2000.

\bibitem{goyal2018global}
P.~Goyal.
\newblock Global attractor for weakly damped, forced m{K}d{V} equation below
  energy space.
\newblock {\em Nagoya Mathematical Journal}, pages 1--33, 2018.

\bibitem{kato1975quasi}
T.~Kato.
\newblock Quasi-linear equations of evolution, with applications to partial
  differential equations.
\newblock In {\em Spectral theory and differential equations}, pages 25--70.
  Springer, 1975.

\bibitem{kenig1996bilinear}
C.~E. Kenig, G.~Ponce, and L.~Vega.
\newblock A bilinear estimate with applications to the {K}d{V} equation.
\newblock {\em Journal of the American Mathematical Society}, 9(2):573--603,
  1996.

\bibitem{killip2019kdv}
R.~Killip and M.~Vi{\c{s}}an.
\newblock {K}d{V} is well-posed in ${H}^{-1}$.
\newblock {\em Annals of Mathematics}, 190(1):249--305, 2019.

\bibitem{oh2013resonant}
S.~Oh et~al.
\newblock Resonant phase-shift and global smoothing of the periodic
  {K}orteweg-de {V}ries equation in low regularity settings.
\newblock {\em Advances in Differential Equations}, 18(7/8):633--662, 2013.

\bibitem{oh2020smoothing}
S.~Oh and A.~G. Stefanov.
\newblock Smoothing and growth bound of periodic generalized {K}orteweg-de
  {V}ries equation.
\newblock {\em arXiv preprint arXiv:2001.08984}, 2020.

\bibitem{temam2012infinite}
R.~Temam.
\newblock {\em Infinite-dimensional dynamical systems in mechanics and
  physics}, volume~68.
\newblock Springer Science \& Business Media, 2012.

\bibitem{tsugawa2004existence}
K.~Tsugawa.
\newblock Existence of the global attractor for weakly damped, forced {K}d{V}
  equation on sobolev spaces of negative index.
\newblock {\em Communications on Pure \& Applied Analysis}, 3(2):301, 2004.

\bibitem{wang2015global}
M.~Wang.
\newblock Global attractor for weakly damped g{K}d{V} equations in higher
  sobolev spaces.
\newblock {\em Discrete \& Continuous Dynamical Systems-A}, 35(8):3799, 2015.

\bibitem{wang2012long}
M.~Wang, D.~Li, C.~Zhang, and Y.~Tang.
\newblock Long time behavior of solutions of g{K}d{V} equations.
\newblock {\em Journal of Mathematical Analysis and Applications},
  390(1):136--150, 2012.

\bibitem{yang2011global}
X.~Yang.
\newblock Global attractor for the weakly damped forced {K}d{V} equation in
  sobolev spaces of low regularity.
\newblock {\em Nonlinear Differential Equations and Applications NoDEA},
  18(3):273--285, 2011.

\end{thebibliography}
\email{ryanm12@illinois.edu}
\end{document}